\pdfoutput=1
\documentclass[a4paper,11pt]{amsart}

\usepackage[margin=1in,includehead,includefoot]{geometry}
\usepackage[utf8]{inputenc}
\usepackage[T1]{fontenc}
\usepackage{lmodern,microtype}
\usepackage{mathtools,amssymb,amsthm}
\usepackage{enumitem}
\usepackage{float}
\usepackage{tikz,pgfplots}
\usetikzlibrary{external}
\usepackage[unicode]{hyperref}
\usepackage{bookmark}

\hypersetup{colorlinks=true,linkcolor=black,citecolor=black,filecolor=black,urlcolor=black}

\makeatletter
\def\@makefntext{\leavevmode\llap{\@makefnmark\kern 3pt}}
\let\citationorig\citation
\def\citation#1{\citationorig{#1}\@for\@tempa:=#1\do{\@ifundefined{cit@\@tempa}{\global\@namedef{cit@\@tempa}{}}{}}}
\let\bibitemorig\bibitem
\def\bibitem#1{\@ifundefined{cit@#1}{\typeout{LaTeX Warning: Unused bibitem `#1'}}{}\bibitemorig{#1}}
\let\old@setaddresses\@setaddresses
\def\@setaddresses{\bigskip{\parindent 0pt\let\scshape\relax\let\ttfamily\relax\old@setaddresses}}
\makeatother

\def\periodsf{\spacefactor 3000 \space}

\newtheorem{theorem}{Theorem}
\newtheorem{lemma}[theorem]{Lemma}
\newtheorem*{explicit}{\explicittitle}
\newenvironment{theorem*}[1]{\begingroup\def\explicittitle{#1}\begin{explicit}}{\end{explicit}\endgroup}
\let\explicittitle\undefined

\renewenvironment{enumerate}{\begin{enumorig}[label=\textup{(\arabic*)}, noitemsep, topsep=3pt plus 3pt, leftmargin=*]}{\end{enumorig}}

\renewenvironment{itemize}{\begin{itemorig}[label=\textbullet, noitemsep, topsep=3pt plus 3pt, labelsep=.6em, labelindent=.2em, leftmargin=*]}{\end{itemorig}}

\let\int\undefined
\DeclareMathOperator{\int}{int}
\DeclareMathOperator{\ext}{ext}

\def\famC{\mathcal{C}}
\def\famF{\mathcal{F}}
\def\famG{\mathcal{G}}
\def\famH{\mathcal{H}}
\def\famP{\mathcal{P}}
\def\famS{\mathcal{S}}
\def\famT{\mathcal{T}}
\def\famU{\mathcal{U}}
\def\famX{\mathcal{X}}
\def\famY{\mathcal{Y}}

\def\setN{\mathbb{N}}
\def\setR{\mathbb{R}}

\let\leq\leqslant
\let\geq\geqslant
\let\setminus\smallsetminus
\let\Omega\varOmega
\let\Theta\varTheta

\allowdisplaybreaks

\hypersetup{
  pdftitle={Coloring curves that cross a fixed curve},
  pdfauthor={Alexandre Rok, Bartosz Walczak},
}

\title{Coloring curves that cross a fixed curve}
\author{Alexandre Rok\and Bartosz Walczak}

\address[Alexandre Rok]{Department of Mathematics, Ben-Gurion University of the Negev, Be'er Sheva 84105, Israel}
\email{\href{mailto:rok@math.bgu.ac.il}{rok@math.bgu.ac.il}}

\address[Bartosz Walczak]{Department of Theoretical Computer Science, Faculty of Mathematics and Computer Science, Jagiellonian University, Kraków, Poland}
\email{\href{mailto:walczak@tcs.uj.edu.pl}{walczak@tcs.uj.edu.pl}}

\thanks{A preliminary version of this paper appeared in: \href{http://doi.org/10.4230/LIPIcs.SoCG.2017.56}{Boris Aronov and Matthew~J. Katz (eds.), \emph{33rd International Symposium on Computational Geometry (SoCG 2017)}, vol.~77 of \emph{Leibniz International Proceedings in Informatics (LIPIcs)}, no.~56, pp.~1--15, Leibniz-Zentrum für Informatik, Dagstuhl, 2017}.}
\thanks{Alexandre Rok was partially supported by Israel Science Foundation grant 1136/12.
Bartosz Walczak was partially supported by National Science Center of Poland grant 2015/17/D/ST1/00585.}

\begin{document}

\begin{abstract}
We prove that for every integer $t\geq 1$, the class of intersection graphs of curves in the plane each of which crosses a fixed curve in at least one and at most $t$ points is $\chi$-bounded.
This is essentially the strongest $\chi$-boundedness result one can get for this kind of graph classes.
As a corollary, we prove that for any fixed integers $k\geq 2$ and $t\geq 1$, every $k$-quasi-planar topological graph on $n$ vertices with any two edges crossing at most $t$ times has $O(n\log n)$ edges.
\end{abstract}

\maketitle

\section{Introduction}

\subsection*{Overview}

A \emph{curve} is a homeomorphic image of the real interval $[0,1]$ in the plane.
The \emph{intersection graph} of a family of curves has these curves as vertices and the intersecting pairs of curves as edges.
Combinatorial and algorithmic aspects of intersection graphs of curves, known as \emph{string graphs}, have been attracting researchers for decades.
A significant part of this research has been devoted to understanding classes of string graphs that are \emph{$\chi$-bounded}, which means that every graph $G$ in the class satisfies $\chi(G)\leq f(\omega(G))$ for some function $f\colon\setN\to\setN$.
Here, $\chi(G)$ and $\omega(G)$ denote the chromatic number and the clique number (the maximum size of a clique) of $G$, respectively.
Recently, Pawlik et~al.\ \cite{PKK+13,PKK+14} proved that the class of all string graphs is not $\chi$-bounded.
However, all known constructions of string graphs with small clique number and large chromatic number require a lot of freedom in placing curves around in the plane.

What restrictions on placement of curves lead to $\chi$-bounded classes of intersection graphs?
McGuinness \cite{McG96,McG00} proposed studying families of curves that cross a fixed curve \emph{exactly once}.
This initiated a series of results culminating in the proof that the class of intersection graphs of such families is indeed $\chi$-bounded \cite{RW}.
By contrast, the class of intersection graphs of curves each crossing a fixed curve \emph{at least once} is equal to the class of all string graphs and therefore is not $\chi$-bounded.
We prove an essentially farthest possible generalization of the former result, allowing curves to cross the fixed curve \emph{at least once and at most\/ $t$ times}, for any bound $t$.

\begin{theorem}
\label{thm:curves}
For every integer\/ $t\geq 1$, the class of intersection graphs of curves each crossing a fixed curve in at least one and at most\/ $t$ points is\/ $\chi$-bounded.
\end{theorem}

Additional motivation for Theorem~\ref{thm:curves} comes from its application to bounding the number of edges in so-called $k$-quasi-planar graphs, which we discuss at the end of this introduction.

\subsection*{Context}

Colorings of intersection graphs of geometric objects have been investigated since the 1960s, when Asplund and Grünbaum \cite{AG60} proved that intersection graphs of axis-parallel rectangles in the plane satisfy $\chi=O(\omega^2)$ and conjectured that the class of intersection graphs of axis-parallel boxes in $\setR^d$ is $\chi$-bounded for every integer $d\geq 1$.
A few years later Burling \cite{Bur65} discovered a surprising construction of triangle-free intersection graphs of axis-parallel boxes in $\setR^3$ with arbitrarily large chromatic number.
Since then, the upper bound of $O(\omega^2)$ and the trivial lower bound of $\Omega(\omega)$ on the maximum possible chromatic number of a rectangle intersection graph have been improved only in terms of multiplicative constants \cite{Hen98,Kos04}.

Another classical example of a $\chi$-bounded class of geometric intersection graphs is provided by circle graphs---intersection graphs of chords of a fixed circle.
Gyárfás \cite{Gya85} proved that circle graphs satisfy $\chi=O(\omega^24^\omega)$.
The best known upper and lower bounds on the maximum possible chromatic number of a circle graph are $O(2^\omega)$ \cite{KK97} and $\Omega(\omega\log\omega)$ \cite{Kos88,Kos04}.

McGuinness \cite{McG96,McG00} proposed investigating the problem in a setting that allows much more general geometric shapes but restricts the way how they are arranged in the plane.
In \cite{McG96}, he proved that the class of intersection graphs of L-shapes crossing a fixed horizontal line is $\chi$-bounded.
Families of L-shapes in the plane are \emph{simple}, which means that any two members of the family intersect in at most one point.
McGuinness \cite{McG00} also showed that triangle-free intersection graphs of simple families of curves each crossing a fixed line in exactly one point have bounded chromatic number.
Further progress in this direction was made by Suk \cite{Suk14}, who proved that simple families of $x$-monotone curves crossing a fixed vertical line give rise to a $\chi$-bounded class of intersection graphs, and by Lasoń et~al.\ \cite{LMPW14}, who reached the same conclusion without assuming that the curves are $x$-monotone.
Finally, in \cite{RW}, we proved that the class of intersection graphs of curves each crossing a fixed line in exactly one point is $\chi$-bounded.
These results remain valid if the fixed straight line is replaced by a fixed curve \cite{SW15}.

The class of string graphs is not $\chi$-bounded.
Pawlik et~al.\ \cite{PKK+13,PKK+14} showed that Burling's construction for boxes in $\setR^3$ can be adapted to provide a construction of triangle-free intersection graphs of straight-line segments (or geometric shapes of various other kinds) with chromatic number growing as fast as $\Theta(\log\log n)$ with the number of vertices $n$.
It was further generalized to a construction of string graphs with clique number $\omega$ and chromatic number $\Theta_\omega((\log\log n)^{\omega-1})$ \cite{KW}.
The best known upper bound on the chromatic number of string graphs in terms of the number of vertices is $\smash[t]{(\log n)^{O(\log\omega)}}$, proved by Fox and Pach \cite{FP14} using a separator theorem for string graphs due to Matoušek \cite{Mat14}.
For intersection graphs of segments and, more generally, $x$-monotone curves, upper bounds of the form $\chi=O_\omega(\log n)$ follow from the above-mentioned results in \cite{Suk14} and \cite{RW} via recursive halving.
Upper bounds of the form $\chi=O_\omega(\smash[t]{(\log\log n)^{f(\omega)}})$ (for some function $f\colon\setN\to\setN$) are known for very special classes of string graphs: rectangle overlap graphs \cite{KPW15,KW} and subtree overlap graphs \cite{KW}.
The former still allow the triangle-free construction with $\chi=\Theta(\log\log n)$ and the latter the construction with $\chi=\Theta_\omega((\log\log n)^{\omega-1})$.

\subsection*{Quasi-planarity}

A \emph{topological graph} is a graph with a fixed curvilinear drawing in the plane.
For $k\geq 2$, a \emph{$k$-quasi-planar graph} is a topological graph with no $k$ pairwise crossing edges.
In particular, a $2$-quasi-planar graph is just a planar graph.
It is conjectured that $k$-quasi-planar graphs with $n$ vertices have $O_k(n)$ edges \cite{BMP-book,PSS96}.
For $k=2$, this asserts a well-known property of planar graphs.
The conjecture is also verified for $k=3$ \cite{AAP+97,PRT06} and $k=4$ \cite{Ack09}, but it remains open for $k\geq 5$.
The best known upper bounds on the number of edges in a $k$-quasi-planar graph are $\smash[t]{n(\log n)^{O(\log k)}}$ in general \cite{FP12,FP14}, $O_k(n\log n)$ for the case of $x$-monotone edges \cite{Val97}, $O_k(n\log n)$ for the case that any two edges intersect at most once \cite{SW15}, and $\smash[t]{2^{\alpha(n)^\nu}}n\log n$ for the case that any two edges intersect in at most $t$ points, where $\alpha$ is the inverse Ackermann function and $\nu$ depends on $k$ and $t$ \cite{SW15}.
We apply Theorem~\ref{thm:curves} to improve the last bound to $O_{k,t}(n\log n)$.

\begin{theorem}
\label{thm:k-quasi-planar}
Every\/ $k$-quasi-planar topological graph\/ $G$ on\/ $n$ vertices such that any two edges of\/ $G$ intersect in at most\/ $t$ points has at most\/ $\mu_{k,t}n\log n$ edges, where\/ $\mu_{k,t}$ depends only on\/ $k$ and\/ $t$.
\end{theorem}

\noindent
The proof follows the same lines as the proof in \cite{SW15} for the case $t=1$ (see Section~\ref{sec:k-quasi-planar}).

\section{Proof of Theorem~\ref{thm:curves}}

\subsection*{Setup}

We let $\setN$ denote the set of positive integers.
Graph-theoretic terms applied to a family of curves $\famF$ have the same meaning as applied to the intersection graph of $\famF$.
In particular, the \emph{chromatic number} of $\famF$, denoted by $\chi(\famF)$, is the minimum number of colors in a \emph{proper coloring} of $\famF$ (a coloring that distinguishes pairs of intersecting curves), and the \emph{clique number} of $\famF$, denoted by $\omega(\famF)$, is the maximum size of a \emph{clique} in $\famF$ (a set of pairwise intersecting curves in $\famF$).

\begin{theorem*}{Theorem~\ref{thm:curves}}[rephrased]
For every\/ $t\in\setN$, there is a non-decreasing function\/ $f_t\colon\setN\to\setN$ with the following property: for any fixed curve\/ $c_0$, every family\/ $\famF$ of curves each intersecting\/ $c_0$ in at least one and at most\/ $t$ points satisfies\/ $\chi(\famF)\leq f_t(\omega(\famF))$.
\end{theorem*}

We do not state any explicit bound on the function $f_t$ above, because it highly depends on the bound on the function $f$ in Theorem~\ref{thm:CSS} (one of our main tools), and no explicit bound on that function is provided in \cite{CSS}.
We assume (implicitly) that the intersection points of all curves $c\in\famF$ with $c_0$ considered in Theorem~\ref{thm:curves} are distinct and each of them is a \emph{proper crossing}, which means that $c$ passes from one to the other side of $c_0$ in a sufficiently small neighborhood of the intersection point.
This assumption is without loss of generality, as it can be guaranteed by appropriate small perturbations of the curves that do not influence the intersection graph.

\subsection*{Initial reduction}

We start by reducing Theorem~\ref{thm:curves} to a somewhat simpler and more convenient setting.
We fix a horizontal line in the plane and call it the \emph{baseline}.
The upper closed half-plane determined by the baseline is denoted by $H^+$.
A \emph{$1$-curve} is a curve in $H^+$ that has one endpoint (called the \emph{basepoint} of the $1$-curve) on the baseline and does not intersect the baseline in any other point.
Intersection graphs of $1$-curves are known as \emph{outerstring graphs} and form a $\chi$-bounded class of graphs---this result, due to the authors, is the starting point of the proof of Theorem~\ref{thm:curves}.

\begin{theorem}[\cite{RW}]
\label{thm:outerstring}
There is a non-decreasing function\/ $f_0\colon\setN\to\setN$ such that every family\/ $\famF$ of\/ $1$-curves satisfies\/ $\chi(\famF)\leq f_0(\omega(\famF))$.
\end{theorem}

An \emph{even-curve} is a curve that has both endpoints above the baseline and a positive even number of intersection points with the baseline, each of which is a proper crossing.
For $t\in\setN$, a \emph{$2t$-curve} is an even-curve that intersects the baseline in exactly $2t$ points.
A \emph{basepoint} of an even-curve $c$ is an intersection point of $c$ with the baseline.
Like above, we assume (implicitly, without loss of generality) that the basepoints of all even-curves in any family that we consider are distinct.
Every even-curve $c$ determines two $1$-curves---the two parts of $c$ from an endpoint to the closest basepoint along $c$.
They are called the \emph{$1$-curves of\/ $c$} and denoted by $L(c)$ and $R(c)$ so that the basepoint of $L(c)$ lies to the left of the basepoint of $R(c)$ on the baseline (see Figure~\ref{fig:even-curve}).
A family $\famF$ of even-curves is an \emph{$LR$-family} if every intersection between two curves $c_1,c_2\in\famF$ is an intersection between $L(c_1)$ and $R(c_2)$ or between $L(c_2)$ and $R(c_1)$.
The main effort in this paper goes to proving the following statement on $LR$-families of even-curves.

\begin{figure}[t]
\centering
\begin{tikzpicture}[scale=.9,yscale=.7]
  \draw[name path=b] (0,0)--(10.2,0);
  \path[name path=c] plot[smooth,tension=.72] coordinates {(2.8,2) (2.7,1.2) (2.4,-0.7) (3.8,-0.6) (4,2.5) (1.4,2.5) (1.2,-1.4) (4.1,-2) (5.8,1.1) (7.2,1) (7.9,-1.4) (9.2,-1) (8.8,2.2) (5.4,2.9)};
  \draw[thick,intersection segments={of=c and b,sequence=L1}];
  \draw[thick,dashed,intersection segments={of=c and b,sequence={L2--L3--L4--L5--L6}}];
  \draw[thick,intersection segments={of=c and b,sequence=L7}];
  \path[decorate,decoration={markings,mark=at position .6 with {\node[above left] {$L(c)$};}},intersection segments={of=c and b,sequence=L1}];
  \path[decorate,decoration={markings,mark=at position .5 with {\node[above] {$M(c)$};}},intersection segments={of=c and b,sequence=L4}];
  \path[decorate,decoration={markings,mark=at position .3 with {\node[below left] {$R(c)$};}},intersection segments={of=c and b,sequence=L7}];
  \draw[decorate,decoration={brace,mirror,amplitude=5pt,raise=3pt},name intersections={of=c and b}] (intersection-1)--(intersection-6) node[midway,below=8pt] {$I(c)$};
\end{tikzpicture}
\caption{$L(c)$, $R(c)$, $M(c)$ (all the dashed part), and $I(c)$ for a $6$-curve $c$}
\label{fig:even-curve}
\end{figure}
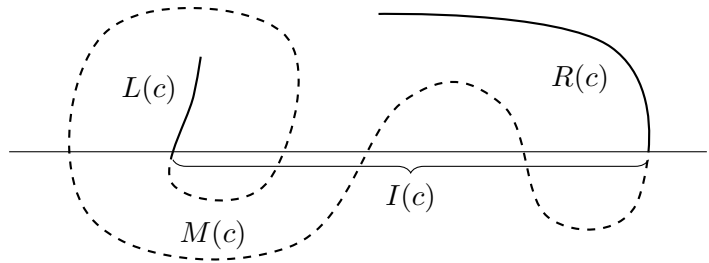

\begin{theorem}
\label{thm:even-curves}
There is a non-decreasing function\/ $f\colon\setN\to\setN$ such that every\/ $LR$-family\/ $\famF$ of even-curves satisfies\/ $\chi(\famF)\leq f(\omega(\famF))$.
\end{theorem}

Theorem~\ref{thm:even-curves} makes no assumption on the maximum number of intersection points of an even-curve with the baseline.
We derive Theorem~\ref{thm:curves} from Theorem~\ref{thm:even-curves} in two steps, first proving the following lemma, and then showing that Theorem~\ref{thm:curves} is essentially a special case of it.

\begin{lemma}
\label{lem:2t-curves}
For every\/ $t\in\setN$, there is a non-decreasing function\/ $f_t\colon\setN\to\setN$ such that every family\/ $\famF$ of\/ $2t$-curves no two of which intersect below the baseline satisfies\/ $\chi(\famF)\leq f_t(\omega(\famF))$.
\end{lemma}

\begin{proof}[Proof of Lemma~\ref{lem:2t-curves} from Theorem~\ref{thm:even-curves}]
The proof goes by induction on $t$.
Let $f_0$ and $f$ be the functions claimed by Theorem~\ref{thm:outerstring} and Theorem~\ref{thm:even-curves}, respectively, and let $f_t(k)=f_{t-1}^2(k)f(k)$ for $t\geq 1$ and $k\in\setN$.
We establish the base case for $t=1$ and the induction step for $t\geq 2$ simultaneously.
Namely, fix an integer $t\geq 1$, and let $\famF$ be as in the statement of the lemma.
For every $2t$-curve $c\in\famF$, enumerate the endpoints and basepoints of $c$ as $p_0(c),\ldots,p_{2t+1}(c)$ in their order along $c$ so that $p_0(c)$ and $p_1(c)$ are the endpoints of $L(c)$ while $p_{2t}(c)$ and $p_{2t+1}(c)$ are the endpoints of $R(c)$.
Build two families of curves $\famF_1$ and $\famF_2$ putting the part of $c$ from $p_0(c)$ to $p_{2t-1}(c)$ to $\famF_1$ and the part of $c$ from $p_2(c)$ to $p_{2t+1}(c)$ to $\famF_2$ for every $c\in\famF$.
If $t=1$, then $\famF_1$ and $\famF_2$ are families of $1$-curves.
If $t\geq 2$, then $\famF_1$ and $\famF_2$ are equivalent to families of $2(t-1)$-curves, because the curve in $\famF_1$ or $\famF_2$ obtained from a $2t$-curve $c\in\famF$ can be shortened a little at $p_{2t-1}(c)$ or $p_2(c)$, respectively, losing that basepoint but no intersection points with other curves.
Therefore, by Theorem~\ref{thm:outerstring} or the induction hypothesis, we have $\chi(\famF_k)\leq f_{t-1}(\omega(\famF_k))\leq f_{t-1}(\omega(\famF))$ for $k\in\{1,2\}$.
For $c\in\famF$ and $k\in\{1,2\}$, let $\phi_k(c)$ be the color of the curve obtained from $c$ in an optimal proper coloring of $\famF_k$.
Every subfamily of $\famF$ on which $\phi_1$ and $\phi_2$ are constant is an $LR$-family and therefore, by Theorem~\ref{thm:even-curves} and monotonicity of $f$, has chromatic number at most $f(\omega(\famF))$.
We conclude that $\chi(\famF)\leq\chi(\famF_1)\chi(\famF_2)f(\omega(\famF))\leq f_{t-1}^2(\omega(\famF))f(\omega(\famF))=f_t(\omega(\famF))$.
\end{proof}

A \emph{closed curve} is a homeomorphic image of a unit circle in the plane.
For a closed curve $\gamma$, the Jordan curve theorem asserts that the set $\setR^2\setminus\gamma$ consists of two arc-connected components, one of which is bounded and denoted by $\int\gamma$ and the other is unbounded and denoted by $\ext\gamma$.

\begin{proof}[Proof of Theorem~\ref{thm:curves} from Theorem~\ref{thm:even-curves}]
We elect to present this proof in an intuitive rather than rigorous way.
Let $\famF$ be a family of curves each intersecting $c_0$ in at least one and at most $t$ points.
Let $\gamma_0$ be a closed curve surrounding $c_0$ very closely so that $\gamma_0$ intersects every curve in $\famF$ in exactly $2t$ points (winding if necessary to increase the number of intersections) and all endpoints of curves in $\famF$ and intersection points of pairs of curves in $\famF$ lie in $\ext\gamma_0$.
We apply geometric inversion to obtain an equivalent family of curves $\famF'$ and a closed curve $\gamma_0'$ with the same properties except that all endpoints of curves in $\famF'$ and intersection points of pairs of curves in $\famF'$ lie in $\int\gamma_0'$.
It follows that some part of $\gamma_0'$ lies in the unbounded component of $\setR^2\setminus\bigcup\famF'$.
We ``cut'' $\gamma_0'$ there and ``unfold'' it into the baseline, transforming $\famF'$ into an equivalent family $\famF''$ of $2t$-curves all endpoints of which and intersection points of pairs of which lie above the baseline.
The ``equivalence'' of $\famF$, $\famF'$, and $\famF''$ means in particular that the intersection graphs of $\famF$, $\famF'$, and $\famF''$ are isomorphic, so the theorem follows from Lemma~\ref{lem:2t-curves} (and thus Theorem~\ref{thm:even-curves}).
\end{proof}

A statement analogous to Theorem~\ref{thm:even-curves} fails for families of objects each consisting of two $1$-curves only, without the ``middle part'' connecting them.
Specifically, we define a \emph{double-curve} as a set $X\subset H^+$ that is a union of two disjoint $1$-curves, denoted by $L(X)$ and $R(X)$ so that the basepoint of $L(X)$ lies to the left of the basepoint of $R(X)$, and we call a family $\famX$ of double-curves an \emph{$LR$-family} if every intersection between two double-curves $X_1,X_2\in\famX$ is an intersection between $L(X_1)$ and $R(X_2)$ or between $L(X_2)$ and $R(X_1)$.

\begin{theorem}
\label{thm:construction}
There exist triangle-free\/ $LR$-families of double-curves with arbitrarily large chromatic number.
\end{theorem}

The proof of Theorem~\ref{thm:construction} is an easy adaptation of the construction from \cite{PKK+13,PKK+14} and is presented in detail in Section~\ref{sec:construction}.
The rest of this section is devoted to the proof of Theorem~\ref{thm:even-curves}.

\subsection*{Overview of the proof of Theorem~\ref{thm:even-curves}}

Recall the assertion of Theorem~\ref{thm:even-curves}: the $LR$-families of even-curves are $\chi$-bounded.
The proof is quite long and technical, so we find it useful to provide a high-level overview of its structure.
The proof will be presented via a series of reductions.
First, we will reduce Theorem~\ref{thm:even-curves} to the following statement (Lemma~\ref{lem:2-curves}): the $LR$-families of $2$-curves are $\chi$-bounded.
This statement will be proved by induction on the clique number.
Specifically, we will prove the following as the induction step: if every $LR$-family of $2$-curves $\famF$ with $\omega(\famF)\leq k-1$ satisfies $\chi(\famF)\leq\xi$, then every $LR$-family of $2$-curves $\famF$ with $\omega(\famF)\leq k$ satisfies $\chi(\famF)\leq\zeta$, where $\zeta$ is a constant depending only on $k$ and $\xi$.
The only purpose of the induction hypothesis is to guarantee that if $\omega(\famF)\leq k$ and $c\in\famF$, then the family of $2$-curves in $\famF\setminus\{c\}$ that intersect $c$ has chromatic number at most $\xi$.
For notational convenience, $LR$-families of $2$-curves with the latter property will be called \emph{$\xi$-families}.
We will thus reduce the problem to the following statement (Lemma~\ref{lem:xi}): the $\xi$-families are $\chi$-bounded, where the $\chi$-bounding function depends on $\xi$.

We will deal with $\xi$-families via a series of technical lemmas of the following general form: every $\xi$-family with sufficiently large chromatic number contains a specific configuration of curves.
Two kinds of such configurations are particularly important: (a) a large clique, and (b) a $2$-curve $c$ and a subfamily $\famF'$ with large chromatic number such that all basepoints of $2$-curves in $\famF'$ lie between the basepoints of $c$.
At the core of the argument are the proofs that
\begin{itemize}
\item every $\xi$-family with sufficiently large chromatic number contains (a) or (b) (Lemma~\ref{lem:encompass}),
\item assuming the above, every $\xi$-family with sufficiently large chromatic number contains (a).
\end{itemize}
Combined, they complete the argument.
Since the two proofs are almost identical, we introduce one more reduction---to \emph{$(\xi,h)$-families} (Lemma~\ref{lem:(xi,h)}).
A $(\xi,h)$-family is just a $\xi$-family that satisfies an additional technical condition that allows us to deliver both proofs at once.

\subsection*{More notation and terminology}

Let $\prec$ denote the left-to-right order of points on the baseline ($p_1\prec p_2$ means that $p_1$ is to the left of $p_2$).
For convenience, we also use the notation $\prec$ for curves intersecting the baseline ($c_1\prec c_2$ means that every basepoint of $c_1$ is to the left of every basepoint of $c_2$) and for families of such curves ($\famC_1\prec\famC_2$ means that $c_1\prec c_2$ for any $c_1\in\famC_1$ and $c_2\in\famC_2$).
For a family $\famC$ of curves intersecting the baseline (even-curves or $1$-curves) and two $1$-curves $x$ and $y$, let $\famC(x,y)=\{c\in\famC\colon x\prec c\prec y\}$ or $\famC(x,y)=\{c\in\famC\colon y\prec c\prec x\}$ depending on whether $x\prec y$ or $y\prec x$.
For a family $\famC$ of curves intersecting the baseline and a segment $I$ on the baseline, let $\famC(I)$ denote the family of curves in $\famC$ with all basepoints on $I$.

For an even-curve $c$, let $M(c)$ denote the subcurve of $c$ connecting the basepoints of $L(c)$ and $R(c)$, and let $I(c)$ denote the segment on the baseline connecting the basepoints of $L(c)$ and $R(c)$ (see Figure~\ref{fig:even-curve}).
For a family $\famF$ of even-curves, let $L(\famF)=\{L(c)\colon c\in\famF\}$, $R(\famF)=\{R(c)\colon c\in\famF\}$, and $I(\famF)$ denote the minimal segment on the baseline that contains $I(c)$ for every $c\in\famF$.

A \emph{cap-curve} is a curve in $H^+$ that has both endpoints on the baseline and does not intersect the baseline in any other point.
It follows from the Jordan curve theorem that for every cap-curve $\gamma$, the set $H^+\setminus\gamma$ consists of two arc-connected components, one of which is bounded and denoted by $\int\gamma$ and the other is unbounded and denoted by $\ext\gamma$.

\subsection*{Reduction to \texorpdfstring{$LR$}{LR}-families of \texorpdfstring{$2$}{2}-curves}

We will reduce Theorem~\ref{thm:even-curves} to the following statement on $LR$-families of $2$-curves, which is essentially a special case of Theorem~\ref{thm:even-curves}.

\begin{lemma}
\label{lem:2-curves}
There is a non-decreasing function\/ $f\colon\setN\to\setN$ such that every\/ $LR$-family\/ $\famF$ of\/ $2$-curves satisfies\/ $\chi(\famF)\leq f(\omega(\famF))$.
\end{lemma}

A \emph{component} of a family of $1$-curves $\famS$ is an arc-connected component of $\bigcup\famS$ (the union of all curves in $\famS$).
The following easy but powerful observation reuses an idea from \cite{LMPW14,McG00,Suk14}.

\begin{lemma}
\label{lem:planar}
For every\/ $LR$-family of even-curves\/ $\famF$, if\/ $\famF^\star$ is the family of curves\/ $c\in\famF$ such that\/ $L(c)$ and\/ $R(c)$ lie in different components of\/ $L(\famF)\cup R(\famF)$, then\/ $\chi(\famF^\star)\leq 4$.
\end{lemma}

\begin{proof}
Let $G$ be an auxiliary graph where the vertices are the components of $L(\famF)\cup R(\famF)$ and the edges are the pairs $V_1V_2$ of components such that there is a curve $c\in\famF^\star$ with $L(c)\subseteq V_1$ and $R(c)\subseteq V_2$ or $L(c)\subseteq V_2$ and $R(c)\subseteq V_1$.
Since $\famF$ is an $LR$-family, the curves in $\famF^\star$ can intersect only within the components of $L(\famF)\cup R(\famF)$.
It follows that $G$ is planar and thus $4$-colorable.
Fix a proper $4$-coloring of $G$, and assign the color of a component $V$ to every curve $c\in\famF^\star$ with $L(c)\subseteq V$.
For any $c_1,c_2\in\famF^\star$, if $L(c_1)$ and $R(c_2)$ intersect, then $L(c_1)$ and $R(c_2)$ lie in the same component $V_1$ while $L(c_2)$ lies in a component $V_2$ such that $V_1V_2$ is an edge of $G$, so $c_1$ and $c_2$ are assigned different colors.
The coloring of $\famF^\star$ is therefore proper.
\end{proof}

\begin{proof}[Proof of Theorem~\ref{thm:even-curves} from Lemma~\ref{lem:2-curves}]
We show that $\chi(\famF)\leq f(\omega(\famF))+4$, where $f$ is the function claimed by Lemma~\ref{lem:2-curves}.
We have $\famF=\famF_1\cup\famF_2$, where $\famF_1=\{c\in\famF\colon L(c)$ and $R(c)$ lie in the same component of $L(\famF)\cup R(\famF)\}$ and $\famF_2=\{c\in\famF\colon L(c)$ and $R(c)$ lie in different components of $L(\famF)\cup R(\famF)\}$.
Lemma~\ref{lem:planar} yields $\chi(\famF_2)\leq 4$.
It remains to show that $\chi(\famF_1)\leq f(\omega(\famF))$.

Let $c_1,c_2\in\famF_1$.
We claim that the intervals $I(c_1)$ and $I(c_2)$ are nested or disjoint.
Suppose they are neither nested nor disjoint.
The components of $L(\famF)\cup R(\famF)$ are disjoint from the curves of the form $M(c)$ with $c\in\famF$ except at common basepoints.
For $k\in\{1,2\}$, since $L(c_k)$ and $R(c_k)$ belong to the same component of $L(\famF)\cup R(\famF)$, the basepoints of $L(c_k)$ and $R(c_k)$ can be connected by a cap-curve $\gamma_k$ disjoint from $M(c)$ for every $c\in\famF$ except at the endpoints of $M(c)$ when $c=c_k$.
We assume (without loss of generality) that $\gamma_1$ and $\gamma_2$ intersect in a finite number of points and each of their intersection points is a proper crossing.
Since the intervals $I(c_1)$ and $I(c_2)$ are neither nested nor disjoint, the basepoints of $L(c_2)$ and $R(c_2)$ lie one in $\int\gamma_1$ and the other in $\ext\gamma_1$.
This implies that $\gamma_1$ and $\gamma_2$ intersect in an odd number of points, by the Jordan curve theorem.
For $k\in\{1,2\}$, let $\tilde\gamma_k$ be the closed curve obtained as the union of $\gamma_k$ and $M(c_k)$.
It follows that $\tilde\gamma_1$ and $\tilde\gamma_2$ intersect in an odd number of points and each of their intersection points is a proper crossing, which is a contradiction to the Jordan curve theorem.

Transform $\famF_1$ into a family of $2$-curves $\famF_1'$ replacing the part $M(c)$ of every $2$-curve $c\in\famF_1$ by the lower semicircle connecting the endpoints of $M(c)$.
Since the intervals $I(c)$ with $c\in\famF_1$ are pairwise nested or disjoint, these semicircles are pairwise disjoint.
Consequently, $\famF_1'$ is an $LR$-family.
Since the intersection graphs of $\famF_1$ and $\famF_1'$ are isomorphic, Lemma~\ref{lem:2-curves} implies $\chi(\famF_1)=\chi(\famF_1')\leq f(\omega(\famF_1'))\leq f(\omega(\famF))$.
\end{proof}

\subsection*{Reduction to \texorpdfstring{$\xi$}{ξ}-families}

For $\xi\in\setN$, a \emph{$\xi$-family} is an $LR$-family of $2$-curves $\famF$ with the following property: for every $2$-curve $c\in\famF$, the family of $2$-curves in $\famF\setminus\{c\}$ that intersect $c$ has chromatic number at most $\xi$.
We reduce Lemma~\ref{lem:2-curves} to the following statement on $\xi$-families.

\begin{lemma}
\label{lem:xi}
For any\/ $\xi,k\in\setN$, there is a constant\/ $\zeta\in\setN$ such that every\/ $\xi$-family\/ $\famF$ with\/ $\omega(\famF)\leq k$ satisfies\/ $\chi(\famF)\leq\zeta$.
\end{lemma}

\begin{proof}[Proof of Lemma~\ref{lem:2-curves} from Lemma~\ref{lem:xi}]
Let $f(1)=1$.
For $k\geq 2$, let $f(k)$ be the constant claimed by Lemma~\ref{lem:xi} such that every $f(k-1)$-family $\famF$ with $\omega(\famF)\leq k$ satisfies $\chi(\famF)\leq f(k)$.
Let $k=\omega(\famF)$, and proceed by induction on $k$ to prove $\chi(\famF)\leq f(k)$.
Clearly, if $k=1$, then $\chi(\famF)=1$.
For the induction step, assume $k\geq 2$.
For every $c\in\famF$, the family of $2$-curves in $\famF\setminus\{c\}$ that intersect $c$ has clique number at most $k-1$ and therefore, by the induction hypothesis, has chromatic number at most $f(k-1)$.
That is, $\famF$ is an $f(k-1)$-family, and the definition of $f$ yields $\chi(\famF)\leq f(k)$.
\end{proof}

\subsection*{Dealing with \texorpdfstring{$\xi$}{ξ}-families}

First, we establish the following special case of Lemma~\ref{lem:xi}.

\begin{lemma}
\label{lem:nested}
For every\/ $\xi\in\setN$, every\/ $\xi$-family\/ $\famF$ with\/ $\bigcap_{c\in\famF}I(c)\neq\emptyset$ satisfies\/ $\chi(\famF)\leq 4\xi+4$.
\end{lemma}

The proof of Lemma~\ref{lem:nested} is essentially the same as the proof of Lemma~19 in \cite{SW15}.
We need the following elementary lemma, which was also used in various forms in \cite{LMPW14,McG96,McG00,RW,Suk14}.
We include its proof, as we will later extend it when proving Lemma~\ref{lem:between}.

\begin{lemma}[McGuinness {\cite[Lemma 2.1]{McG96}}]
\label{lem:mcguinness}
Let\/ $G$ be a graph, $\prec$ be a total order on the vertices of\/ $G$, and\/ $\alpha,\beta\in\setN$.
If\/ $\chi(G)>(2\beta+2)\alpha$, then\/ $G$ has an induced subgraph\/ $H$ such that\/ $\chi(H)>\alpha$ and\/ $\chi(G(u,v))>\beta$ for every edge\/ $uv$ of\/ $H$.
In particular, if\/ $\chi(G)>2\beta+2$, then\/ $G$ has an edge\/ $uv$ with\/ $\chi(G(u,v))>\beta$.
Here, $G(u,v)$ denotes the subgraph of\/ $G$ induced on the vertices strictly between\/ $u$ and\/ $v$ in the order\/ $\prec$.
\end{lemma}

\begin{proof}
Let $G[U]$ denote the subgraph of $G$ induced on a set of vertices $U$.
Partition the vertices of $G$ into subsets $V_0\prec\cdots\prec V_n$ so that $\chi(G[V_i])=\beta+1$ for $0\leq i<n$ and $\chi(G[V_n])\leq\beta+1$.
This is done greedily, by processing the vertices of $G$ in the order $\prec$, adding them to $V_0$ until $\chi(G[V_0])=\beta+1$, then adding them to $V_1$ until $\chi(G[V_1])=\beta+1$, and so on.
For $0\leq i\leq n$, a proper $(\beta+1)$-coloring of $G[V_i]$ yields a partition of $V_i$ into color classes $V_i^1,\ldots,V_i^{\beta+1}$ that are independent sets in $G$.
Let $r\in\{1,\ldots,\beta+1\}$ be such that $\chi(G[\bigcup_{i=0}^nV_i^r])$ is maximized.
It follows that $\chi(G[\bigcup_{i=0}^nV_i^r])\geq\chi(G)/(\beta+1)>2\alpha$ and thus $\chi(G[\bigcup_{i\text{ even}}V_i^r])>\alpha$ or $\chi(G[\bigcup_{i\text{ odd}}V_i^r])>\alpha$.
Let $H=G[\bigcup_{i\text{ even}}V_i^r]$ or $H=G[\bigcup_{i\text{ odd}}V_i^r]$ accordingly, so that $\chi(H)>\alpha$.
Now, if $uv$ is an edge of $H$, then $u\in V_k^r$ and $v\in V_\ell^r$ for two distinct indices $k,\ell\in\{0,\ldots,n\}$ of the same parity (because each $V_i^r$ is an independent set in $G$), and therefore $G[V_i]$ is a subgraph of $G(u,v)$ for every (at least one) index $i\in\{1,\ldots,n-1\}$ strictly between $k$ and $\ell$, witnessing $\chi(G(u,v))>\beta$.
\end{proof}

\begin{proof}[Proof of Lemma~\ref{lem:nested}]
Suppose $\chi(\famF)>4\xi+4$.
Since $\bigcap_{c\in\famF}I(c)\neq\emptyset$, the $2$-curves in $\famF$ can be enumerated as $c_1,\ldots,c_n$ so that $L(c_1)\prec\cdots\prec L(c_n)\prec R(c_n)\prec\cdots\prec R(c_1)$, where $n=\lvert\famF\rvert$.
Lemma~\ref{lem:mcguinness} applied to the intersection graph of $\famF$ and the order $c_1,\ldots,c_n$ provides two indices $i,j\in\{1,\ldots,n\}$ such that the $2$-curves $c_i$ and $c_j$ intersect and $\chi\bigl(\{c_{i+1},\ldots,c_{j-1}\}\bigr)>2\xi+1$.
Assume $L(c_i)$ and $R(c_j)$ intersect; the argument for the other case is analogous.
There is a cap-curve $\nu\subseteq L(c_i)\cup R(c_j)$ connecting the basepoints of $L(c_i)$ and $R(c_j)$.
Every curve intersecting $\nu$ intersects $c_i$ or $c_j$.
Since $\famF$ is a $\xi$-family, the $2$-curves in $\{c_{i+1},\ldots,c_{j-1}\}$ that intersect $c_i$ have chromatic number at most $\xi$, and so do those that intersect $c_j$.
Every $2$-curve $c_k\in\{c_{i+1},\ldots,c_{j-1}\}$ not intersecting $\nu$ satisfies $L(c_k)\subset\int\nu$ and $R(c_k)\subset\ext\nu$, so these $2$-curves are pairwise disjoint.
We conclude that $\chi\bigl(\{c_{i+1},\ldots,c_{j-1}\}\bigr)\leq 2\xi+1$, which is a contradiction.
\end{proof}

Lemma~\ref{lem:mcguinness} easily implies that every family of $2$-curves $\famF$ with $\chi(\famF)>(2\beta+2)^2\alpha$ contains a subfamily $\famH$ with $\chi(\famH)>\alpha$ such that $\chi(\famF(L(c_1),L(c_2)))>\beta$ and $\chi(\famF(R(c_1),R(c_2)))>\beta$ for any two intersecting $2$-curves $c_1,c_2\in\famH$.
This is considerably strengthened by the following lemma.
Its proof is based on the same general idea as the proof of Lemma~\ref{lem:mcguinness} presented above.

\begin{lemma}
\label{lem:between}
For every\/ $\xi\in\setN$, there is a function\/ $f\colon\setN\times\setN\to\setN$ with the following property: for any\/ $\alpha,\beta\in\setN$ and every\/ $\xi$-family\/ $\famF$ with\/ $\chi(\famF)>f(\alpha,\beta)$, there is a subfamily\/ $\famH\subseteq\famF$ such that\/ $\chi(\famH)>\alpha$ and\/ $\chi(\famF(x,y))>\beta$ for any two intersecting\/ $1$-curves\/ $x,y\in L(\famH)\cup R(\famH)$.
\end{lemma}

\begin{proof}
Let $f(\alpha,\beta)=(2\beta+12\xi+20)\alpha$.
Let $\famF$ be a $\xi$-family with $\chi(\famF)>f(\alpha,\beta)$.
Construct a sequence of points $p_0\prec\cdots\prec p_{n+1}$ on the baseline with the following properties:
\begin{itemize}
\item the points $p_0,\ldots,p_{n+1}$ are distinct from all basepoints of $2$-curves in $\famF$,
\item $p_0$ lies to the left of and $p_{n+1}$ lies to the right of all basepoints of $2$-curves in $\famF$,
\item $\chi(\famF(p_ip_{i+1}))=\beta+1$ for $0\leq i<n$ and $\chi(\famF(p_np_{n+1}))\leq\beta+1$.
\end{itemize}
This is done greedily, by first choosing $p_1$ so that $\chi(\famF(p_0p_1))=\beta+1$, then choosing $p_2$ so that $\chi(\famF(p_1p_2))=\beta+1$, and so on.
For $0\leq i\leq j\leq n$, let $\famF_{i,j}=\{c\in\famF\colon p_i\prec L(c)\prec p_{i+1}$ and $p_j\prec R(c)\prec p_{j+1}\}$.
In particular, $\famF_{i,i}=\famF(p_ip_{i+1})$ for $0\leq i\leq n$.
Since $\famF=\bigcup_{0\leq i\leq j\leq n}\famF_{i,j}$ and $\chi(\famF)\geq(2\beta+12\xi+20)\alpha$, at least one of the following inequalities holds:
\begin{equation*}
\textstyle\chi\bigl(\bigcup_{i=0}^n\famF_{i,i}\bigr)>(2\beta+2)\alpha,\quad\chi\bigl(\bigcup_{i=0}^{n-1}\famF_{i,i+1}\bigr)>(12\xi+12)\alpha,\quad\chi\bigl(\bigcup_{i=0}^{n-2}\bigcup_{j=i+2}^n\famF_{i,j}\bigr)>6\alpha.
\end{equation*}
In each case, we claim we can find a subfamily $\famH\subseteq\famF$ such that any two intersecting $1$-curves $x\in R(\famH)$ and $y\in L(\famH)$ satisfy $x\in R(\famF_{i,j})$ and $y\in L(\famF_{r,s})$, where $0\leq i\leq j\leq n$, $0\leq r\leq s\leq n$, and $\lvert j-r\rvert\geq 2$.
Then, we have $\chi(\famF(x,y))\geq\chi(\famF(\smash[b]{p_{\min(j,r)+1}p_{\max(j,r)}}))\geq\beta+1$, as required.

Suppose $\chi\bigl(\bigcup_{i=0}^n\famF_{i,i}\bigr)>(2\beta+2)\alpha$.
We have $\chi(\famF_{i,i})\leq\beta+1$ for $0\leq i\leq n$.
Color the $2$-curves in each $\famF_{i,i}$ properly using the same set of $\beta+1$ colors on $\famF_{i,i}$ and $\famF_{r,r}$ whenever $i\equiv r\pmod{2}$, thus partitioning the family $\bigcup_{i=0}^n\famF_{i,i}$ into $2\beta+2$ color classes.
Since $\chi\bigl(\bigcup_{i=0}^n\famF_{i,i}\bigr)>(2\beta+2)\alpha$, at least one such color class $\famH\subseteq\bigcup_{i=0}^n\famF_{i,i}$ satisfies $\chi(\famH)>\alpha$.
To conclude, for any two intersecting $1$-curves $x\in R(\famH)$ and $y\in L(\famH)$, we have $x\in R(\famF_{i,i})$ and $y\in L(\famF_{r,r})$ for some distinct indices $i,r\in\{0,\ldots,n\}$ such that $i\equiv r\pmod{2}$ and thus $\lvert i-r\rvert\geq 2$.

Now, suppose $\chi\bigl(\bigcup_{i=0}^{n-1}\famF_{i,i+1}\bigr)>(12\xi+12)\alpha$.
By Lemma~\ref{lem:nested}, we have $\chi(\famF_{i,i+1})\leq 4\xi+4$ for $0\leq i\leq n-1$.
Color the $2$-curves in every $\famF_{i,i+1}$ properly using the same set of $4\xi+4$ colors on $\famF_{i,i+1}$ and $\famF_{r,r+1}$ whenever $i\equiv r\pmod{3}$, thus partitioning the family $\bigcup_{i=0}^{n-1}\famF_{i,i+1}$ into $12\xi+12$ color classes.
At least one such color class $\famH$ satisfies $\chi(\famH)>\alpha$.
To conclude, for any two intersecting $1$-curves $x\in R(\famH)$ and $y\in L(\famH)$, we have $x\in R(\famF_{i,i+1})$ and $y\in L(\famF_{r,r+1})$ for some distinct indices $i,r\in\{0,\ldots,n-1\}$ such that $i\equiv r\pmod{3}$ and thus $\lvert i+1-r\rvert\geq 2$.

Finally, suppose $\chi\bigl(\bigcup_{i=0}^{n-2}\bigcup_{j=i+2}^n\famF_{i,j}\bigr)>6\alpha$.
It follows that $\chi\bigl(\bigcup_{i\in I}\bigcup_{j=i+2}^n\famF_{i,j}\bigr)>3\alpha$, where $I=\{i\in\{0,\ldots,n-2\}\colon i\equiv 0\pmod{2}\}$ or $I=\{i\in\{0,\ldots,n-2\}\colon i\equiv 1\pmod{2}\}$.
Consider an auxiliary graph $G$ with vertex set $I$ and edge set $\{ij\colon i,j\in I$, $i<j$, and $\famF_{i,j-1}\cup\famF_{i,j}\neq\emptyset\}$.
If there were two edges $i_1j_1$ and $i_2j_2$ in $G$ with $i_1<i_2<j_1<j_2$, then their witnessing $2$-curves, one from $\smash[b]{\famF_{i_1,j_1-1}}\cup\smash[b]{\famF_{i_1,j_1}}$ and the other from $\smash[b]{\famF_{i_2,j_2-1}}\cup\smash[b]{\famF_{i_2,j_2}}$, would intersect below the baseline, which is impossible.
This shows that $G$ is an outerplanar graph, and thus $\chi(G)\leq 3$.
Fix a proper $3$-coloring of $G$, and use the color of $i$ on every $2$-curve in $\bigcup_{j=i+2}^n\famF_{i,j}$ for every $i\in I$, partitioning the family $\bigcup_{i\in I}\bigcup_{j=i+2}^n\famF_{i,j}$ into $3$ color classes.
At least one such color class $\famH$ satisfies $\chi(\famH)>\alpha$.
To conclude, for any two intersecting $1$-curves $x\in R(\famH)$ and $y\in L(\famH)$, we have $x\in R(\famF_{i,j})$ and $y\in L(\famF_{r,s})$ for some indices $i,r\in I$, $j\in\{i+2,\ldots,n\}$, and $s\in\{r+2,\ldots,n\}$ such that $j\notin\{r-1,r\}$ (otherwise $ir$ would be an edge of $G$), $j\neq r+1$ (otherwise two $2$-curves, one from $\famF_{i,r+1}$ and one from $\famF_{r,s}$, would intersect below the baseline), and thus $\lvert j-r\rvert\geq 2$.
\end{proof}

Lemma~2 in \cite{RW} asserts that for every family of $1$-curves $\famS$ with at least one intersecting pair, there are a cap-curve $\gamma$ and a subfamily $\famT\subseteq\famS$ with $\chi(\famT)\geq\chi(\famS)/2$ such that every $1$-curve in $\famT$ is entirely contained in $\int\gamma$ and intersects some $1$-curve in $\famS$ that intersects $\gamma$ (equivalently, $\ext\gamma$).
The proof follows a standard idea, originally due to Gyárfás \cite{Gya85}, to choose $\famT$ as one of the sets of $1$-curves at a fixed distance from an appropriately chosen $1$-curve in the intersection graph of $\famS$.
However, this method fails to imply an analogous statement for $2$-curves.
We will need a more powerful tool---part of the recent series of works on induced subgraphs that must be present in graphs with sufficiently large chromatic number.

\begin{theorem}[Chudnovsky, Scott, Seymour {\cite[Theorem 1.8]{CSS}}]
\label{thm:CSS}
There is a function\/ $f\colon\setN\to\setN$ with the following property: for every\/ $\alpha\in\setN$, every string graph\/ $G$ with\/ $\chi(G)>f(\alpha)$ contains a vertex\/ $v$ such that\/ $\chi(G^2_v)>\alpha$, where\/ $G^2_v$ denotes the subgraph of\/ $G$ induced on the vertices within distance at most\/ $2$ from\/ $v$.
\end{theorem}

The special case of Theorem~\ref{thm:CSS} for triangle-free intersection graphs of curves any two of which intersect in at most one point was proved earlier by McGuinness \cite[Theorem 5.3]{McG01}.

\begin{figure}[t]
\centering
\begin{tikzpicture}[scale=.9,yscale=.7]
  \draw[dotted,red,fill=red!10] plot[smooth,tension=.7] coordinates {(5.4,0) (5.6,1) (6.4,3) (4.6,3.6) (2.8,1.6) (2,2.2) (3.6,4.8) (7.4,5.2) (13.6,5.2) (14,0)};
  \draw[red] plot[smooth,tension=.75] coordinates {(1.9,0.8) (2.2,-0.6) (3.1,-0.5) (3.4,1.5) (2.6,3.4) (0.6,4)};
  \draw[red] plot[smooth,tension=.7] coordinates {(4.2,2.2) (5.4,1.8) (7.7,1.4) (8.5,-1.4) (10.5,-2.4) (12.7,-1.3) (12.8,2.2) (11.1,4.9)};
  \draw[name path=c0,blue] plot[smooth,tension=.75] coordinates {(1.6,4.8) (1,-0.6) (3.8,-1.2) (5,2.9)};
  \draw[name path=b] (0,0)--(15,0);
  \draw[name path=c1] plot[smooth,tension=.65] coordinates {(6.4,0.8) (6.4,-0.6) (7.2,-0.6) (7.6,1.3) (8.7,2.1) (10.5,1.9) (11.1,2.8) (10.5,4)};
  \draw[name path=c2] plot[smooth,tension=.75] coordinates {(2.6,2.2) (4.7,4.5) (8.3,3.5) (9.1,-1) (11.6,-1.1) (11.9,2.2)};
  \draw[name path=c3] plot[smooth,tension=.75] coordinates {(10.9,0.9) (10.8,-0.7) (9.7,-0.4) (10,2.2) (11.8,3.8) (12.6,4.6)};
  \path[decorate,decoration={markings,mark=at position .5 with {\node[below] {$c^\star$};}},intersection segments={of=c0 and b,sequence=L2}];
  \path[decorate,decoration={markings,mark=at position .5 with {\node[below] {$c_1$};}},intersection segments={of=c1 and b,sequence=L2}];
  \path[decorate,decoration={markings,mark=at position .6 with {\node[below left] {$c_2$};}},intersection segments={of=c2 and b,sequence=L1}];
  \path[decorate,decoration={markings,mark=at position .15 with {\node[right] {$c_3$};}},intersection segments={of=c3 and b,sequence=L3}];
  \node[above] at (7.4,5.2) {$\gamma$};
  \node at (9.4,4.4) {$\int\gamma$};
\end{tikzpicture}
\caption{Illustration for Lemma~\ref{lem:supports}: $\famG=\{c_1,c_2,c_3\}$}
\label{fig:supports}
\end{figure}
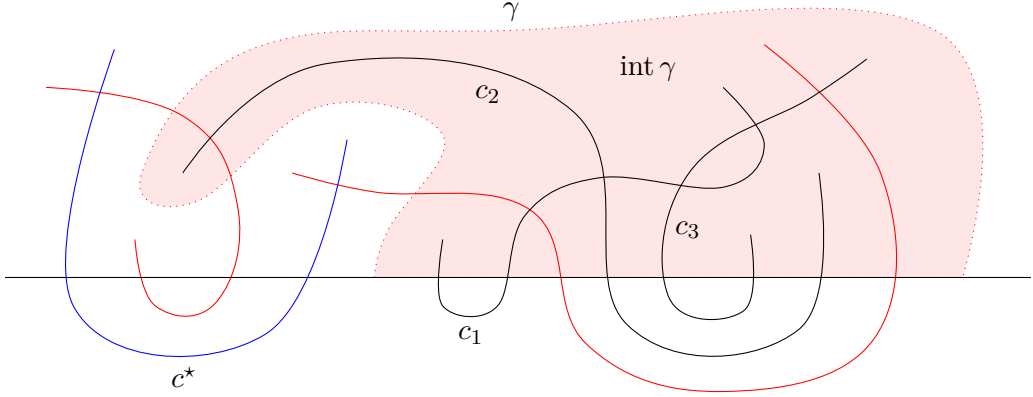

\begin{lemma}[see Figure~\ref{fig:supports}]
\label{lem:supports}
For every\/ $\xi\in\setN$, there is a function\/ $f\colon\setN\to\setN$ with the following property: for every\/ $\alpha\in\setN$ and every\/ $\xi$-family\/ $\famF$ with\/ $\chi(\famF)>f(\alpha)$, there are a cap-curve\/ $\gamma$ and a subfamily\/ $\famG\subseteq\famF$ with\/ $\chi(\famG)>\alpha$ such that every\/ $2$-curve\/ $c\in\famG$ satisfies\/ $L(c),R(c)\subset\int\gamma$ and intersects some\/ $2$-curve in\/ $\famF$ that intersects\/ $\ext\gamma$.
\end{lemma}

\begin{proof}
Let $f(\alpha)=f_1(3\alpha+5\xi+5)$, where $f_1$ is the function claimed by Theorem~\ref{thm:CSS}.
Let $\famF$ be a $\xi$-family with $\chi(\famF)>f(\alpha)$.
It follows that there is a $2$-curve $c^\star\in\famF$ such that the family of curves within distance at most $2$ from $c^\star$ in the intersection graph of $\famF$ has chromatic number greater than $3\alpha+5\xi+5$.
For $k\in\{1,2\}$, let $\famF_k$ be the $2$-curves in $\famF$ at distance exactly $k$ from $c^\star$ in the intersection graph of $\famF$.
We have $\chi(\{c^\star\}\cup\famF_1\cup\famF_2)>3\alpha+5\xi+5$ (by Theorem~\ref{thm:CSS}) and $\chi(\famF_1)\leq\xi$ (because $\famF$ is a $\xi$-family), so $\chi(\famF_2)>3\alpha+4\xi+4$.
We have $\famF_2=\famG_1\cup\famG_2\cup\famG_3\cup\famG_4$, where
\begin{align*}
\famG_1&=\{c\in\famF_2\colon L(c)\prec R(c)\prec L(c^\star)\prec R(c^\star)\}, & \famG_2&=\{c\in\famF_2\colon L(c^\star)\prec L(c)\prec R(c)\prec R(c^\star)\}, \\
\famG_3&=\{c\in\famF_2\colon L(c^\star)\prec R(c^\star)\prec L(c)\prec R(c)\}, & \famG_4&=\{c\in\famF_2\colon L(c)\prec L(c^\star)\prec R(c^\star)\prec R(c)\}.
\end{align*}
Since $\chi(\famF_2)>3\alpha+4\xi+4$ and $\chi(\famG_4)\leq 4\xi+4$ (by Lemma~\ref{lem:nested}), we have $\chi(\famG_k)>\alpha$ for some $k\in\{1,2,3\}$.
Since neither basepoint of $c^\star$ lies on the segment $I(\famG_k)$, there is a cap-curve $\gamma$ with $L(c^\star),R(c^\star)\subset\ext\gamma$ and $L(c),R(c)\subset\int\gamma$ for all $c\in\famG_k$.
The lemma follows with $\famG=\famG_k$.
\end{proof}

\subsection*{Reduction to \texorpdfstring{$(\xi,h)$}{(ξ,h)}-families}

For $\xi\in\setN$ and a function $h\colon\setN\to\setN$, a \emph{$(\xi,h)$-family} is a $\xi$-family $\famF$ with the following additional property: for every $\alpha\in\setN$ and every subfamily $\famG\subseteq\famF$ with $\chi(\famG)>h(\alpha)$, there is a subfamily $\famH\subseteq\famG$ with $\chi(\famH)>\alpha$ such that every $2$-curve in $\famF$ with a basepoint on $I(\famH)$ has both basepoints on $I(\famG)$.
We will prove the following lemma.

\begin{lemma}
\label{lem:(xi,h)}
For any\/ $\xi,k\in\setN$ and any function\/ $h\colon\setN\to\setN$, there is a constant\/ $\zeta\in\setN$ such that every\/ $(\xi,h)$-family\/ $\famF$ with\/ $\omega(\famF)\leq k$ satisfies\/ $\chi(\famF)\leq\zeta$.
\end{lemma}

The notion of a $(\xi,h)$-family and Lemma~\ref{lem:(xi,h)} provide a convenient abstraction of what is needed to prove the next lemma and then to derive Lemma~\ref{lem:xi} from the next lemma.

\begin{lemma}
\label{lem:encompass}
For any\/ $\xi,k\in\setN$, there is a function\/ $f\colon\setN\to\setN$ such that for every\/ $\alpha\in\setN$, every\/ $\xi$-family\/ $\famF$ with\/ $\omega(\famF)\leq k$ and\/ $\chi(\famF)>f(\alpha)$ contains a\/ $2$-curve\/ $c$ with\/ $\chi(\famF(I(c)))>\alpha$.
\end{lemma}

\begin{proof}[Proof of Lemma~\ref{lem:encompass} from Lemma~\ref{lem:(xi,h)}]
For $\alpha\in\setN$, let $h_\alpha\colon\setN\ni\beta\mapsto\beta+2\alpha+2\in\setN$, and let $f(\alpha)$ be the constant claimed by Lemma~\ref{lem:(xi,h)} such that every $(\xi,h_\alpha)$-family $\famF$ with $\omega(\famF)\leq k$ satisfies $\chi(\famF)\leq f(\alpha)$.
Let $\famF$ be a $\xi$-family with $\omega(\famF)\leq k$ and $\chi(\famF(I(c)))\leq\alpha$ for every $c\in\famF$.
We show that $\famF$ is a $(\xi,h_\alpha)$-family, which then implies $\chi(\famF)\leq f(\alpha)$.
To this end, consider a subfamily $\famG\subseteq\famF$ with $\chi(\famG)>h_\alpha(\beta)$ for some $\beta\in\setN$.
Take $\famG_L,\famG_R\subseteq\famG$ greedily so that $L(\famG_L)\prec L(\famG\setminus\famG_L)$, $\chi(\famG_L)=\alpha+1$, $R(\famG\setminus\famG_R)\prec R(\famG_R)$, and $\chi(\famG_R)=\alpha+1$.
Let $\famH=\famG\setminus(\famG_L\cup\famG_R)$.
It follows that $\chi(\famH)\geq\chi(\famG)-\chi(\famG_L)-\chi(\famG_R)>h_\alpha(\beta)-2\alpha-2=\beta$.
If there is a $2$-curve $c\in\famF$ with one basepoint on $I(\famH)$ and the other basepoint not on $I(\famG)$, then $\famG_L\subseteq\famF(I(c))$ or $\famG_R\subseteq\famF(I(c))$, so $\chi(\famF(I(c)))\geq\alpha+1$, which is a contradiction.
Therefore, every $2$-curve in $\famF$ with a basepoint on $I(\famH)$ has both basepoints on $I(\famG)$.
This shows that $\famF$ is a $(\xi,h_\alpha)$-family.
\end{proof}

\begin{proof}[Proof of Lemma~\ref{lem:xi} from Lemma~\ref{lem:(xi,h)}]
Let $h$ be the function claimed by Lemma~\ref{lem:encompass} for $\xi$ and $k$.
Let $\zeta$ be the constant claimed by Lemma~\ref{lem:(xi,h)} for $\xi$, $k$, and $h$.
Let $\famF$ be a $\xi$-family with $\omega(\famF)\leq k$.
We show that $\famF$ is a $(\xi,h)$-family, which then implies $\chi(\famF)\leq\zeta$.
To this end, consider a subfamily $\famG\subseteq\famF$ with $\chi(\famG)>h(\alpha)$ for some $\alpha\in\setN$.
Lemma~\ref{lem:encompass} yields a $2$-curve $c\in\famG$ such that $\chi(\famG(I(c)))>\alpha$.
Every $2$-curve in $\famF$ with a basepoint on $I(c)$ has both basepoints on $I(c)$, otherwise it would intersect $c$ below the baseline.
Therefore, the condition on $\famF$ being a $(\xi,h)$-family is satisfied with $\famH=\famG(I(c))$.
\end{proof}

\subsection*{Dealing with \texorpdfstring{$(\xi,h)$}{(ξ,h)}-families}

The rest of this section is devoted to the proof of Lemma~\ref{lem:(xi,h)}.
Its structure and principal ideas are based on those of the proof of Theorem~\ref{thm:outerstring} presented in \cite{RW}.
For each forthcoming lemma, we provide a reference to its counterpart in \cite{RW}.

A \emph{skeleton} is a pair $(\gamma,\famU)$ such that $\gamma$ is a cap-curve and $\famU$ is a family of pairwise disjoint $1$-curves each of which has one endpoint (other than the basepoint) on $\gamma$ and all the remaining part in $\int\gamma$ (see Figure~\ref{fig:skeleton}).
For a family of $1$-curves $\famS$, a skeleton $(\gamma,\famU)$ is an \emph{$\famS$-skeleton} if every $1$-curve in $\famU$ is a subcurve of some $1$-curve in $\famS$.
A family of $2$-curves $\famG$ is \emph{supported} by a skeleton $(\gamma,\famU)$ if every $2$-curve $c\in\famG$ satisfies $L(c),R(c)\subset\int\gamma$ and intersects some $1$-curve in $\famU$.
A family of $2$-curves $\famH$ is \emph{supported from outside} by a family of $1$-curves $\famS$ if every $2$-curve in $\famH$ intersects some $1$-curve in $\famS$ and every $1$-curve in $\famS$ satisfies $s\prec\famH$ or $\famH\prec s$.

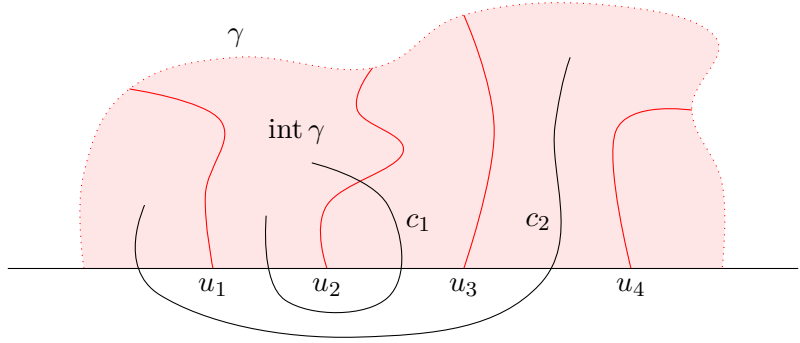
\begin{figure}[t]
\centering
\begin{tikzpicture}[yscale=.7]
  \draw[name path=g,red,dotted,fill=red!10] plot[smooth,tension=.7] coordinates {(1,0) (1,1.8) (1.6,3.4) (3,4) (4.8,3.8) (6,4.8) (7.8,5) (9.3,4.4) (9,3) (9.4,1.6) (9.4,0)};
  \draw[red] plot[smooth,tension=.7] coordinates {(2.7,0) (2.6,1.4) (2.8,2.8) (1.6,3.4)};
  \draw[red] plot[smooth,tension=.7] coordinates {(4.2,0) (4.2,1.2) (5.2,2.2) (4.6,3) (4.8,3.8)};
  \draw[red] plot[smooth,tension=.7] coordinates {(6,0) (6.4,2.6) (6,4.8)};
  \draw[red] plot[smooth,tension=.7] coordinates {(8.2,0) (8,2.6) (9,3)};
  \draw[name path=b] (0,0)--(10.4,0);
  \draw[name path=c1] plot[smooth,tension=.75] coordinates {(3.4,1) (3.65,-0.6) (5,-0.55) (5,1.2) (4,2)};
  \draw[name path=c2] plot[smooth,tension=.75] coordinates {(1.8,1.2) (2,-0.5) (4.6,-1.3) (7,-0.3) (7.2,2.7) (7.4,4)};
  \node[above] at (3,4) {$\gamma$};
  \node[below] at (2.7,0) {$u_1$};
  \node[below] at (4.2,0) {$u_2$};
  \node[below] at (6,0) {$u_3$};
  \node[below] at (8.2,0) {$u_4$};
  \path[decorate,decoration={markings,mark=at position .3 with {\node[right] {$c_1$};}},intersection segments={of=c1 and b,sequence=L3}];
  \path[decorate,decoration={markings,mark=at position .225 with {\node[left] {$c_2$};}},intersection segments={of=c2 and b,sequence=L3}];
  \node at (3.8,2.6) {$\int\gamma$};
\end{tikzpicture}
\caption{A skeleton $\bigl(\gamma,\{u_1,u_2,u_3,u_4\}\bigr)$, which supports $c_1$ but not $c_2$}
\label{fig:skeleton}
\end{figure}

\begin{lemma}[cf.\ {\cite[Lemma~5]{RW}}]
\label{lem:sideways}
For every\/ $\xi\in\setN$ and every function\/ $h\colon\setN\to\setN$, there is a function\/ $f\colon\setN\times\setN\to\setN$ such that for any\/ $\alpha,\beta\in\setN$, every\/ $(\xi,h)$-family\/ $\famF$ with\/ $\chi(\famF)>f(\alpha,\beta)$ contains at least one of the following configurations:
\begin{itemize}
\item a subfamily\/ $\famG\subseteq\famF$ with\/ $\chi(\famG)>\alpha$ supported by an\/ $L(\famF)$-skeleton or an\/ $R(\famF)$-skeleton,\pagebreak
\item a subfamily\/ $\famH\subseteq\famF$ with\/ $\chi(\famH)>\beta$ supported from outside by a family of\/ $1$-curves\/ $\famS$ such~that\/ $\famS\subseteq L(\famF)$ or\/ $\famS\subseteq R(\famF)$.
\end{itemize}
\end{lemma}

\begin{proof}
Let $f(\alpha,\beta)=f_1(2\alpha+h(2\beta)+4)$, where $f_1$ is the function claimed by Lemma~\ref{lem:supports}.
Let $\famF$ be a $(\xi,h)$-family with $\chi(\famF)>f(\alpha,\beta)$.
Apply Lemma~\ref{lem:supports} to obtain a cap-curve $\gamma$ and a subfamily $\famG\subseteq\famF$ with $\chi(\famG)>2\alpha+h(2\beta)+4$ such that every $2$-curve $c\in\famG$ satisfies $L(c),R(c)\subset\int\gamma$ and intersects some $2$-curve in $\famF_{\ext}$.
Here and further on, $\famF_{\ext}$ denotes the family of $2$-curves in $\famF$ that intersect $\ext\gamma$.
Let $\famU_L$ be the $1$-curves that are subcurves of $1$-curves in $L(\famF)$, have one endpoint (other than the basepoint) on $\gamma$, and have all the remaining part in $\int\gamma$.
Let $\famU_R$ be the analogous subcurves of $1$-curves in $R(\famF)$.
Thus $(\gamma,\famU_L)$ is an $L(\famF)$-skeleton, and $(\gamma,\famU_R)$ is an $R(\famF)$-skeleton.
Let $\famG_L$ be the $2$-curves in $\famG$ that intersect some $1$-curve in $\famU_L$, and let $\famG_R$ be those that intersect some $1$-curve in $\famU_R$.
If $\chi(\famG_L)>\alpha$ or $\chi(\famG_R)>\alpha$, then the first conclusion of the lemma holds.
Thus assume $\chi(\famG_L)\leq\alpha$ and $\chi(\famG_R)\leq\alpha$.
Let $\famG'=\famG\setminus(\famG_L\cup\famG_R)$.
It follows that $\chi(\famG')\geq\chi(\famG)-2\alpha>h(2\beta)+4$.

By Lemma~\ref{lem:planar}, the $2$-curves $c\in\famG'$ such that $L(c)$ and $R(c)$ lie in different components of $L(\famG')\cup R(\famG')$ have chromatic number at most $4$.
Therefore, there is a component $V$ of $L(\famG')\cup R(\famG')$ such that $\chi(\famG'_V)\geq\chi(\famG')-4>h(2\beta)$, where $\famG'_V=\{c\in\famG'\colon L(c),R(c)\subseteq V\}$.
There is a cap-curve $\nu\subseteq V$ connecting the two endpoints of the segment $I(\famG'_V)$.
Suppose there is a $2$-curve $c\in\famF_{\ext}$ with both basepoints on $I(\famG'_V)$.
If $L(c)$ intersects $\ext\gamma$, then the part of $L(c)$ from the basepoint to the first intersection point with $\gamma$, which is a $1$-curve in $\famU_L$, intersects $\nu$ (as $\nu\subseteq V\subset\int\gamma$) and thus a curve in $\famG'$ (as $V$ is a component of $\famG'$); this implies $\famG'\cap\famG_L\neq\emptyset$, which is a contradiction.
An analogous contradiction is reached if $R(c)$ intersects $\ext\gamma$.
This shows that no curve in $\famF_{\ext}$ has both basepoints on $I(\famG'_V)$.

Since $\famF$ is a $(\xi,h)$-family and $\chi(\famG'_V)>h(2\beta)$, there is a subfamily $\famH'\subseteq\famG'_V$ such that $\chi(\famH')>2\beta$ and every $2$-curve in $\famF$ with a basepoint on $I(\famH')$ has the other basepoint on $I(\famG'_V)$.
This and the above imply that no curve in $\famF_{\ext}$ has a basepoint on $I(\famH')$.
Since every curve in $\famH'$ intersects some curve in $\famF_{\ext}$, we have $\famH'=\famH_L\cup\famH_R$, where $\famH_L$ are the $2$-curves in $\famH'$ that intersect some $1$-curve in $L(\famF_{\ext})$ and $\famH_R$ are those that intersect some $1$-curve in $R(\famF_{\ext})$.
Since $\chi(\famH')>2\beta$, we conclude that $\chi(\famH_L)>\beta$ or $\chi(\famH_R)>\beta$ and thus the second conclusion of the lemma holds with $(\famH,\famS)=(\famH_L,L(\famF_{\ext}))$ or $(\famH,\famS)=(\famH_R,R(\famF_{\ext}))$.
\end{proof}

\begin{lemma}[cf.\ {\cite[Lemma~8]{RW}}]
\label{lem:skeleton}
For every\/ $\xi\in\setN$ and every function\/ $h\colon\setN\to\setN$, there is a function\/ $f\colon\setN\to\setN$ such that for every\/ $\alpha\in\setN$, every\/ $(\xi,h)$-family\/ $\famF$ with\/ $\chi(\famF)>f(\alpha)$ contains a subfamily\/ $\famG\subseteq\famF$ with\/ $\chi(\famG)>\alpha$ supported by an\/ $L(\famF)$-skeleton or an\/ $R(\famF)$-skeleton.
\end{lemma}

\begin{proof}
Let $f(\alpha)=f_1(\alpha,f_1(\alpha,f_1(\alpha,4\xi)))$, where $f_1$ is the function claimed by Lemma~\ref{lem:sideways}.
Let $\famF$ be a $(\xi,h)$-family with $\chi(\famF)>f(\alpha)$.
Suppose for the sake of contradiction that every subfamily of $\famF$ supported by an $L(\famF)$-skeleton or an $R(\famF)$-skeleton has chromatic number at most $\alpha$.
Let $\famF_0=\famF$.
Apply Lemma~\ref{lem:sideways} (and the second conclusion thereof) three times to find families $\famF_1$, $\famF_2$, $\famF_3$, $\famS_1$, $\famS_2$, and $\famS_3$ with the following properties:
\begin{itemize}
\item $\famF=\famF_0\supseteq\famF_1\supseteq\famF_2\supseteq\famF_3$,
\item for $1\leq i\leq 3$, we have $\famS_i\subseteq L(\famF_{i-1})$ or $\famS_i\subseteq R(\famF_{i-1})$, and $\famF_i$ is supported from outside by $\famS_i$.
\item $\chi(\famF_1)>f_1(\alpha,f_1(\alpha,4\xi))$, $\chi(\famF_2)>f_1(\alpha,4\xi)$ and\/ $\chi(\famF_3)>4\xi$.
\end{itemize}
There are indices $i$ and $j$ with $1\leq i<j\leq 3$ such that $\famS_i$ and $\famS_j$ are of the same ``type'': either $\famS_i\subseteq L(\famF_{i-1})$ and $\famS_j\subseteq L(\famF_{j-1})$ or $\famS_i\subseteq R(\famF_{i-1})$ and $\famS_j\subseteq R(\famF_{j-1})$.
Assume for the rest of the proof that $\famS_i\subseteq R(\famF_{i-1})$ and $\famS_j\subseteq R(\famF_{j-1})$; the argument for the other case is analogous.

Let $\famS_L=\{s\in\famS_j\colon s\prec\famF_j\}$, $\famS_R=\{s\in\famS_j\colon\famF_j\prec s\}$, $\famF_L$ be the $2$-curves in $\famF_j$ that intersect some $1$-curve in $\famS_L$, and $\famF_R$ be those that intersect some $1$-curve in $\famS_R$.
Thus $\famF_L\cup\famF_R=\famF_j$.
This and $\chi(\famF_j)\geq\chi(\famF_3)>4\xi$ yield $\chi(\famF_L)>2\xi$ or $\chi(\famF_R)>2\xi$.
Assume for the rest of the proof that $\chi(\famF_L)>2\xi$; the argument for the other case is analogous.

Let $\famS_L^{\min}$ be an inclusion-minimal subfamily of $\famS_L$ subject to the condition that $L(c)$ intersects some $1$-curve in $\famS_L^{\min}$ for every $2$-curve $c\in\famF_L$.
Let $s^\star$ be the $1$-curve in $\famS_L^{\min}$ with rightmost basepoint, and let $\famF_L^\star=\{c\in\famF_L\colon L(c)$ intersects $s^\star\}$.
Since $\famF$ is a $\xi$-family, we have $\chi(\famF_L^\star)\leq\xi$.
By minimality of $\famS_L^{\min}$, the family $\famF_L^\star$ contains a $2$-curve $c^\star$ disjoint from every $1$-curve in $\famS_L^{\min}$ other than $s^\star$.
Since $c^\star\in\famF_j\subseteq\famF_i$ and $\famF_i$ is supported from outside by $\famS_i$, there is a $1$-curve $s_i\in\famS_i$ that intersects $L(c^\star)$.
We show that every $2$-curve in $\famF_L\setminus\famF_L^\star$ intersects $s_i$.

Let $c\in\famF_L\setminus\famF_L^\star$, and let $s$ be a $1$-curve in $\famS_L^{\min}$ that intersects $L(c)$.
We have $s\neq s^\star$, as $c\notin\famF_L^\star$.
There is a cap-curve $\nu\subseteq s\cup L(c)$.
Since $s\prec s^\star\prec L(c)$ and $s^\star$ intersects neither $s$ nor $L(c)$, we have $s^\star\subset\int\nu$.
Since $L(c^\star)$ intersects $s^\star$ but neither $s$ nor $L(c)$, we also have $L(c^\star)\subset\int\nu$.
Since $s\in\famS_j\subseteq R(\famF_i)$ and $s_i\prec\famF_i$ or $\famF_i\prec s_i$, the basepoint of $s_i$ lies in $\ext\nu$.
Since $s_i$ intersects $L(c^\star)$ and $L(c^\star)\subset\int\nu$, the $1$-curve $s_i$ intersects $\nu$ and thus $L(c)$.
This shows that every $2$-curve in $\famF_L\setminus\famF_L^\star$ intersects $s_i$.
This and the assumption that $\famF$ is a $\xi$-family yield $\chi(\famF_L\setminus\famF_L^\star)\leq\xi$.
We conclude that $\chi(\famF_L)\leq\chi(\famF_L^\star)+\chi(\famF_L\setminus\famF_L^\star)\leq 2\xi$, which is a contradiction.
\end{proof}

A \emph{chain} of length $n$ is a sequence $\bigl((a_1,b_1),\ldots,(a_n,b_n)\bigr)$ of pairs of $2$-curves such that
\begin{itemize}
\item for $1\leq i\leq n$, the $1$-curves $R(a_i)$ and $L(b_i)$ intersect,
\item for $2\leq i\leq n$, the basepoints of $R(a_i)$ and $L(b_i)$ lie between the basepoints of $R(a_{i-1})$ and $L(b_{i-1})$, and $L(a_i)$ intersects $R(a_1),\ldots,R(a_{i-1})$ or $R(b_i)$ intersects $L(b_1),\ldots,L(b_{i-1})$.
\end{itemize}

\begin{lemma}[cf.\ {\cite[Lemma~11]{RW}}]
\label{lem:chain}
For every\/ $\xi\in\setN$ and every function\/ $h\colon\setN\to\setN$, there is a function\/ $f\colon\setN\to\setN$ such that for every\/ $n\in\setN$, every\/ $(\xi,h)$-family\/ $\famF$ with\/ $\chi(\famF)>f(n)$ contains a chain of length\/ $n$.
\end{lemma}

\begin{proof}[Proof of Lemma~\ref{lem:(xi,h)} from Lemma~\ref{lem:chain}]
Let $\zeta=f(2k+1)$, where $f$ is the function claimed by Lemma~\ref{lem:chain} for $\xi$ and $h$.
Let $\famF$ be a $(\xi,h)$-family with $\chi(\famF)>\zeta$.
By Lemma~\ref{lem:chain}, $\famF$ contains a chain of length $2k+1$.
This chain contains a subchain $\bigl((a_1,b_1),\ldots,(a_{k+1},b_{k+1})\bigr)$ of pairs of the same ``type''---such that $L(a_i)$ intersects $R(a_1),\ldots,R(a_{i-1})$ for $2\leq i\leq k+1$ or $R(b_i)$ intersects $L(b_1),\ldots,L(b_{i-1})$ for $2\leq i\leq k+1$.
This subchain contains a clique $\{a_1,\ldots,a_{k+1}\}$ or $\{b_1,\ldots,b_{k+1}\}$, respectively, which is not possible when $\omega(\famF)\leq k$.
\end{proof}

\begin{proof}[Proof of Lemma~\ref{lem:chain}]
We define the function $f$ by induction.
We set $f(1)=1$; if $\chi(\famF)>1$, then $\famF$ contains two intersecting $2$-curves, which form a chain of length $1$.
For the induction step, fix $n\geq 1$, and assume that $f(n)$ is defined so that every $(\xi,h)$-family $\famH$ with $\chi(\famH)>f(n)$ contains a chain of length $n$.
Let $f_1$ be the function claimed by Lemma~\ref{lem:between} and $f_2$ be the function claimed by Lemma~\ref{lem:skeleton}.
Let
\begin{equation*}
\beta=f_1\bigl(f(n),h(2\xi)+4\xi+2\bigr),\qquad f(n+1)=f_2(f_2(f_2(\beta))).
\end{equation*}
Let $\famF$ be a $(\xi,h)$-family with $\chi(\famF)>f(n+1)$.
We claim that $\famF$ contains a chain of length $n+1$.

Let $\famF_0=\famF$.
Lemma~\ref{lem:skeleton} applied three times provides families of $2$-curves $\famF_1$, $\famF_2$, $\famF_3$ and skeletons $(\gamma_1,\famU_1)$, $(\gamma_2,\famU_2)$, $(\gamma_3,\famU_3)$ with the following properties:
\begin{itemize}
\item $\famF=\famF_0\supseteq\famF_1\supseteq\famF_2\supseteq\famF_3$,
\item for $1\leq i\leq 3$, $(\gamma_i,\famU_i)$ is an $L(\famF_{i-1})$-skeleton or an $R(\famF_{i-1})$-skeleton supporting $\famF_i$,
\item $\chi(\famF_1)>f_2(f_2(\beta))$, $\chi(\famF_2)>f_2(\beta)$, and $\chi(\famF_3)>\beta$.
\end{itemize}
There are indices $i$ and $j$ with $1\leq i<j\leq 3$ such that the skeletons $(\gamma_i,\famU_i)$ and $(\gamma_j,\famU_j)$ are of the same ``type'': either an $L(\famF_{i-1})$-skeleton and an $L(\famF_{j-1})$-skeleton or an $R(\famF_{i-1})$-skeleton and an $R(\famF_{j-1})$-skeleton.
Assume for the rest of the proof that $(\gamma_i,\famU_i)$ is an $L(\famF_{i-1})$-skeleton and $(\gamma_j,\famU_j)$ is an $L(\famF_{j-1})$-skeleton; the argument for the other case is analogous.

By Lemma~\ref{lem:between}, since $\chi(\famF_j)\geq\chi(\famF_3)>\beta$, there is a subfamily $\famH\subseteq\famF_j$ such that $\chi(\famH)>f(n)$ and $\chi(\famF_j(x,y))>h(2\xi)+4\xi+2$ for any two intersecting $1$-curves $x,y\in L(\famH)\cup R(\famH)$.
Since $\chi(\famH)>f(n)$, there is a chain $\bigl((a_1,b_1),\ldots,(a_n,b_n)\bigr)$ of length $n$ in $\famH$.
Let $x$ and $y$ be the $1$-curves $R(a_n)$ and $L(b_n)$ ordered so that $x\prec y$.
Since they intersect, we have $\chi(\famF_j(x,y))>h(2\xi)+4\xi+2$.

Since $\famF_j\subseteq\famF_i$ and $\famF_i$ is supported by $(\gamma_i,\famU_i)$, every $2$-curve in $\famF_j(x,y)$ intersects some $1$-curve in $\famU_i$.
Let $\famG=\{c\in\famF_j(x,y)\colon c$ intersects some $1$-curve in $\famU_i(x,y)\}$.
If a $2$-curve $c\in\famF_j(x,y)$ intersects no $1$-curve in $\famU_i(x,y)$, then $c$ intersects the $1$-curve in $\famU_i$ with rightmost basepoint to the left of the basepoint of $x$ (if such a $1$-curve exists) or the $1$-curve in $\famU_i$ with leftmost basepoint to the right of the basepoint of $y$ (if such a $1$-curve exists).
This and the fact that $\famF$ is a $\xi$-family imply $\chi(\famF_j(x,y)\setminus\famG)\leq 2\xi$ and thus $\chi(\famG)\geq\chi(\famF_j(x,y))-2\xi>h(2\xi)+2\xi+2$.

\begin{figure}[t]
\centering
\begin{tikzpicture}[scale=.9,yscale=.8]
  \useasboundingbox (0,-1.3) rectangle (15,6.7);
  \coordinate (a0) at (0.6,0);
  \coordinate (a1) at (1.6,5);
  \coordinate (a2) at (5.8,6.2);
  \coordinate (a3) at (13,5.6);
  \coordinate (a4) at (14.4,0);
  \coordinate (b0) at (1.6,0);
  \coordinate (b1) at (2.5,3.6);
  \coordinate (b2) at (6,4.6);
  \coordinate (b3) at (12.1,5);
  \coordinate (b4) at (13.4,0);
  \coordinate (u0) at (4.3,0);
  \coordinate (u1) at (4.4,3);
  \coordinate (u2) at (3.8,5.6);
  \coordinate (u3) at (5,7);
  \coordinate (w0) at (10.1,0);
  \coordinate (w1) at (9.9,1.8);
  \coordinate (w2) at (9.5,4);
  \coordinate (w3) at (10.4,5.9);
  \coordinate (w4) at (9.3,7);
  \fill[red!5] plot[smooth,tension=.7] coordinates {(a0) (a1) (a2) (a3) (a4)};
  \fill[blue!10] plot[smooth,tension=.7] coordinates {(b0) (b1) (b2) (b3) (b4)};
  \begin{scope}
  \clip plot[smooth,tension=.7] coordinates {(a0) (a1) (a2) (a3) (a4)}--cycle;
  \fill[opacity=.1] plot[smooth,tension=.7] coordinates {(u0) (u1) (u2) (u3) (w4) (w3) (w2) (w1) (w0)}--cycle;
  \end{scope}
  \draw[name path=a,dotted,red] plot[smooth,tension=.7] coordinates {(a0) (a1) (a2) (a3) (a4)};
  \draw[dotted,blue] plot[smooth,tension=.7] coordinates {(b0) (b1) (b2) (b3) (b4)};
  \path[name path=b] plot[smooth,tension=.7] coordinates {(b0) (b1) (b2) (b3) (b4)}--cycle;
  \path[name path=uw] plot[smooth,tension=.7] coordinates {(u0) (u1) (u2) (u3) (w4) (w3) (w2) (w1) (w0)};
  \draw[red,intersection segments={of=uw and a,sequence=L1}];
  \draw[red,intersection segments={of=uw and a,sequence=L3}];
  \coordinate (v0) at (8.4,0);
  \path[decorate,decoration={markings,mark=at position .52 with {\coordinate (v4);}}] plot[smooth,tension=.7] coordinates {(a0) (a1) (a2) (a3) (a4)};
  \draw[red] plot[smooth,tension=.7] coordinates {(v0) (8.7,1.6) (8.3,3.6) (8.6,5.4) (v4)};
  \path[decorate,decoration={markings,mark=at position .35 with {\coordinate (b1');},mark=at position .49 with {\coordinate (b2');}}] plot[smooth,tension=.7] coordinates {(b0) (b1) (b2) (b3) (b4)};
  \draw[blue] plot[smooth,tension=.7] coordinates {(5.2,0) (5.2,1.8) (5.4,3.6) (b1')};
  \draw[blue] plot[smooth,tension=.7] coordinates {(7.7,0) (7.3,2.7) (b2')};
  \draw plot[smooth,tension=.7] coordinates {(2.6,0) (3.6,2.9) (8.1,3.5) (10.8,3.9) (11.8,4.15)};
  \coordinate (x0) at (3.5,0);
  \draw plot[smooth,tension=.7] coordinates {(x0) (4.4,2.2) (8.6,2.7) (10.6,2.9) (11.3,2.2) (10.6,1.4)};
  \draw plot[smooth,tension=.7] coordinates {(6,0) (6,1) (6.85,1.4)};
  \path[name path=y] plot[smooth,tension=.7] coordinates {(12.8,1.2) (11.9,0.6) (11.8,-0.5) (11.15,-0.7) (10.9,0) (11.3,1.3) (10.6,2.2)};
  \draw[dashed,intersection segments={of=y and b,sequence=L1--L2}];
  \draw[intersection segments={of=y and b,sequence=L3}];
  \draw plot[smooth,tension=.7] coordinates {(12.5,0) (12.3,1.6) (11.95,2.75) (11.3,3.7) (11,4.7)};
  \path[name path=c] plot[smooth,tension=.77] coordinates {(8,1.5) (9.1,0.8305) (9.05,-0.65) (7.45,-0.8) (6.5,0.9) (6.3,3.6) (6.9,5.6)};
  \draw[dashed,intersection segments={of=c and b,sequence=L1--L2}];
  \draw[blue,intersection segments={of=c and b,sequence=L3}];
  \draw[dashed,intersection segments={of=c and b,sequence=L4}];
  \draw (0,0)--(15,0);
  \node at (4.9,5.4) {$K$};
  \node[below] at (x0) {$x$};
  \node[below] at (u0) {$u_L$};
  \node[below] at (v0) {$u$};
  \node[below] at (w0) {$u_R$};
  \draw[decorate,decoration={brace,mirror,amplitude=5pt,raise=3pt}] (5.5,0)--(7.4,0) node[midway,below=8pt] {$I(\famH')$};
  \path[decorate,decoration={markings,mark=at position .54 with {\node[above] {$c^\star$};}},intersection segments={of=c and b,sequence=L2}];
  \path[decorate,decoration={markings,mark=at position .13 with {\node[right] {$u^\star$};}},intersection segments={of=c and b,sequence=L3}];
  \path[decorate,decoration={markings,mark=at position .33 with {\node[left] {$y$};}},intersection segments={of=y and b,sequence=L3}];
  \path[decorate,decoration={markings,mark=at position .8 with {\node[above right] {$\gamma_i$};}}] plot[smooth,tension=.7] coordinates {(a0) (a1) (a2) (a3) (a4)};
  \path[decorate,decoration={markings,mark=at position .8 with {\node[above right,xshift=-1pt] {$\gamma_j$};}}] plot[smooth,tension=.7] coordinates {(b0) (b1) (b2) (b3) (b4)};
\end{tikzpicture}
\caption{Illustration for the proof of Lemma~\ref{lem:chain}}
\label{fig:final-proof}
\end{figure}
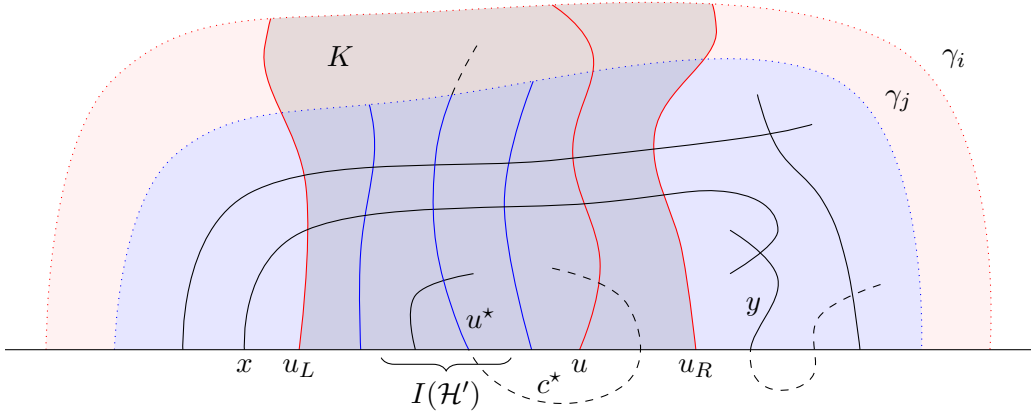

The rest of the argument is illustrated in Figure~\ref{fig:final-proof}.
Let $u_L$ and $u_R$ be the curves in $\famU_i(x,y)$ with leftmost and rightmost basepoints, respectively.
Every $1$-curve in $\famU_i(x,y)$ lies in the closed region $K$ bounded by $u_L$, $u_R$, the segment of the baseline between the basepoints of $u_L$ and $u_R$, and the part of $\gamma_i$ between its intersection points with $u_L$ and $u_R$.
Since $\famF$ is a $\xi$-family, the $2$-curves in $\famG$ intersecting $u_L$ or $u_R$ have chromatic number at most $2\xi$.
Every other $2$-curve $c\in\famG$ satisfies $L(c)\subset K$ or $R(c)\subset K$.
Those for which $L(c)\subset K$ but $R(c)\not\subset K$ satisfy $R(c)\cap K=\emptyset$ and therefore are disjoint from each other.
Similarly, those for which $R(c)\subset K$ but $L(c)\not\subset K$ are disjoint from each other.
Let $\famG'=\{c\in\famG\colon L(c)\subset K$ and $R(c)\subset K\}$.
It follows that $\chi(\famG\setminus\famG')\leq 2\xi+2$ and thus $\chi(\famG')\geq\chi(\famG)-2\xi-2>h(2\xi)$.

Since $\famF$ is a $(\xi,h)$-family, there is a subfamily $\famH'\subseteq\famG'$ with $\chi(\famH')>2\xi$ such that every $2$-curve $c\in\famF$ with a basepoint on $I(\famH')$ satisfies $u_L\prec c\prec u_R$.
Since $\famH'\subseteq\famF_j$ and $\famF_j$ is supported by $(\gamma_j,\famU_j)$, every $2$-curve in $\famH'$ intersects some $1$-curve in $\famU_j$.
If a $2$-curve $c\in\famH'$ intersects no $1$-curve in $\famU_j(I(\famH'))$, then $c$ intersects the $1$-curve in $\famU_j$ with rightmost basepoint to the left of $I(\famH')$ (if such a $1$-curve exists) or the $1$-curve in $\famU_j$ with leftmost basepoint to the right of $I(\famH')$ (if such a $1$-curve exists).
Since $\famF$ is a $\xi$-family, the $2$-curves in $\famH'$ intersecting at least one of these two $1$-curves have chromatic number at most $2\xi$.
Therefore, since $\chi(\famH')>2\xi$, some $2$-curve in $\famH'$ intersects a $1$-curve in $\famU_j(I(\famH'))$.
In particular, the family $\famU_j(I(\famH'))$ is non-empty.

Let $u^\star\in\famU_j(I(\famH'))$.
The $1$-curve $u^\star$ is a subcurve of $L(c^\star)$ for some $2$-curve $c^\star\in\famF_{j-1}$.
The fact that the basepoint of $L(c^\star)$ lies on $I(\famH')$ and the property of $\famH'$ imply $u_L\prec c^\star\prec u_R$.
Since $c^\star\in\famF_{j-1}\subseteq\famF_i$ and $\famF_i$ is supported by $(\gamma_i,\famU_i)$, the $1$-curve $R(c^\star)$ intersects a $1$-curve $u\in\famU_i$, which can be chosen so that $u_L\preceq u\preceq u_R$, because $c^\star\subset\int\gamma_i$ and both basepoints of $c^\star$ lie in $K$.
Let $a_{n+1}=c^\star$ and $b_{n+1}$ be the $2$-curve in $\famF_{i-1}$ such that $u$ is a subcurve of $L(b_{n+1})$.
Thus $x\prec\{R(a_{n+1}),L(b_{n+1})\}\prec y$.
For $1\leq t\leq n$, the facts that the $1$-curves $R(a_t)$ and $L(b_t)$ intersect, they are both contained in $\int\gamma_j$ (as $a_t,b_t\in\famF_j$), the basepoint of $u^\star$ lies between the basepoints of $R(a_t)$ and $L(b_t)$, and $u^\star$ intersects $\gamma_j$ imply that $u^\star$ and therefore $L(a_{n+1})$ intersects $R(a_t)$.
We conclude that $\bigl((a_1,b_1),\ldots,(a_{n+1},b_{n+1})\bigr)$ is a chain of length $n+1$.
\end{proof}

\section{Proof of Theorem~\ref{thm:k-quasi-planar}}
\label{sec:k-quasi-planar}

\begin{lemma}[Fox, Pach, Suk {\cite[Lemma 3.2]{FPS13}}]
\label{lem:decomposable}
For every\/ $t\in\setN$, there is a constant\/ $\nu_t>0$ such that every family of curves\/ $\famF$ any two of which intersect in at most\/ $t$ points has subfamilies\/ $\famF_1,\ldots,\famF_d\subseteq\famF$ (where\/ $d$ is arbitrary) with the following properties:
\begin{itemize}
\item for\/ $1\leq i\leq d$, there is a curve\/ $c_i\in\famF_i$ intersecting all curves in\/ $\famF_i\setminus\{c_i\}$,
\item for\/ $1\leq i<j\leq d$, every curve in\/ $\famF_i$ is disjoint from every curve in\/ $\famF_j$,
\item $\lvert\famF_1\cup\cdots\cup\famF_d\rvert\geq\nu_t{\lvert\famF\rvert}/\log{\lvert\famF\rvert}$.
\end{itemize}
\end{lemma}

\begin{proof}[Proof of Theorem~\ref{thm:k-quasi-planar}]
Let $\famF$ be a family of curves obtained from the edges of $G$ by shortening them slightly so that they do not intersect at the endpoints but all other intersection points are preserved.
If follows that $\omega(\famF)\leq k-1$ (as $G$ is $k$-quasi-planar) and any two curves in $\famF$ intersect in at most $t$ points.
Let $\nu_t$, $\famF_1,\ldots,\famF_d$, and $c_1,\ldots,c_d$ be as claimed by Lemma~\ref{lem:decomposable}.
For $1\leq i\leq d$, since $\omega(\famF_i\setminus\{c_i\})\leq\omega(\famF)-1\leq k-2$, Theorem~\ref{thm:curves} yields $\chi(\famF_i\setminus\{c_i\})\leq f_t(k-2)$.
Thus $\chi(\famF_1\cup\cdots\cup\famF_d)\leq f_t(k-2)+1$.
For every color class $\famC$ in a proper coloring of $\famF_1\cup\cdots\cup\famF_d$ with $f_t(k-2)+1$ colors, the vertices of $G$ and the curves in $\famC$ form a planar topological graph, and thus $\lvert\famC\rvert<3n$.
Thus $\lvert\famF_1\cup\cdots\cup\famF_d\rvert<3(f_t(k-2)+1)n$.
This, the third property in Lemma~\ref{lem:decomposable}, and the fact that $\lvert\famF\rvert<n^2$ yield $\lvert\famF\rvert<3\nu_t^{-1}(f_t(k-2)+1)n\log{\lvert\famF\rvert}<6\nu_t^{-1}(f_t(k-2)+1)n\log n$.
\end{proof}

\section{Proof of Theorem~\ref{thm:construction}}
\label{sec:construction}

\begin{proof}[Proof of Theorem~\ref{thm:construction}]
A \emph{probe} is a section of $H^+$ bounded by two vertical rays starting at the baseline.
We use induction to construct, for every positive integer $k$, an $LR$-family $\famX_k$ of double-curves and a family $\famP_k$ of pairwise disjoint probes with the following properties:
\begin{enumerate}
\item\label{item:shape} every probe in $\famP_k$ is disjoint from $L(X)$ for every double-curve $X\in\famX_k$,
\item\label{item:probe} for every probe $P\in\famP_k$, the double-curves in $\famX_k$ intersecting $P$ are pairwise disjoint,
\item\label{item:omega} $\famX_k$ is triangle-free, that is, $\omega(\famX_k)\leq 2$,
\item\label{item:color} for every proper coloring of $\famX_k$, there is a probe $P\in\famP_k$ such that at least $k$ distinct colors are used on the double-curves in $\famX_k$ intersecting $P$.
\end{enumerate}
This is enough for the proof of theorem, because the last property implies $\chi(\famX_k)\geq k$.
For a pair $(\famX_k,\famP_k)$ satisfying the conditions above and a probe $P\in\famP_k$, let $\famX_k(P)$ denote the set of double-curves in $\famX_k$ intersecting $P$.

For the base case $k=1$, we let $\famX_1=\{X\}$ and $\famP_1=\{P\}$, where $X$ and $P$ look as follows:

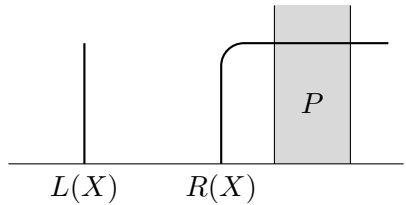
\begin{figure}[H]
\centering
\begin{tikzpicture}
  \useasboundingbox (0,-0.55) rectangle (5.2,2.1);
  \fill[black!15] (3.5,0) rectangle (4.5,2.1);
  \node at (4,0.8) {$P$};
  \draw (0,0)--(5.2,0);
  \draw (3.5,0)--(3.5,2.1) (4.5,0)--(4.5,2.1);
  \draw[thick] (1,0) node[below] {$L(X)$}--(1,1.6);
  \draw[thick] (2.8,0) node[below] {$R(X)$}--(2.8,1.3) arc (180:90:3mm)--(5,1.6);
\end{tikzpicture}
\end{figure}

\noindent
It is clear that the conditions \ref{item:shape}--\ref{item:color} are satisfied.

For the induction step, we assume $k\geq 1$ and construct the pair $(\famX_{k+1},\famP_{k+1})$ from $(\famX_k,\famP_k)$.
Let $(\famX,\famP)$ be a copy of $(\famX_k,\famP_k)$.
For every probe $P\in\famP$, put another copy $(\famX^P,\famP^P)$ of $(\famX_k,\famP_k)$ inside $P$ below the intersections of $P$ with the double-curves in $\famX(P)$.
Then, for every probe $P\in\famP$ and every probe $Q\in\famP^P$, let a double-curve $X^P_Q$ and probes $A^P_Q$ and $B^P_Q$ look as follows:

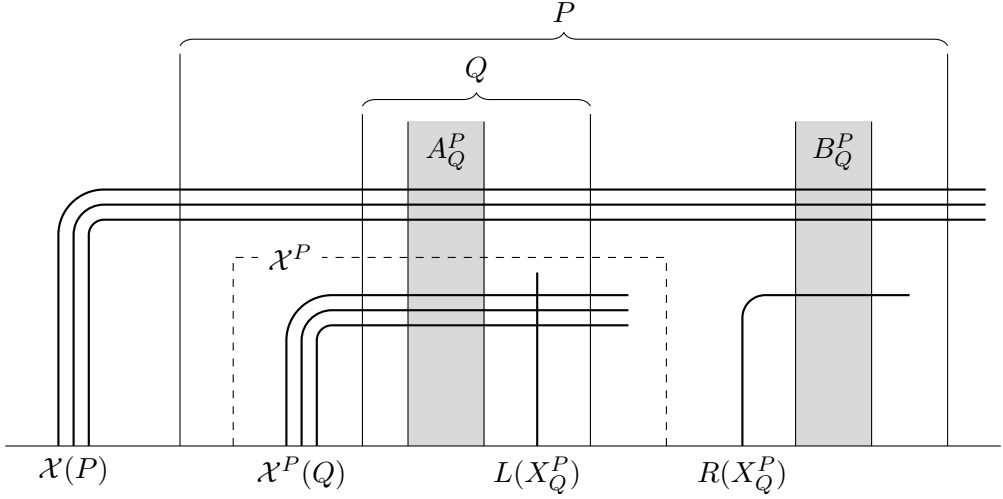
\begin{figure}[H]
\centering
\begin{tikzpicture}
  \fill[black!15] (4.6,0) rectangle (5.6,4.3);
  \fill[black!15] (9.7,0) rectangle (10.7,4.3);
  \node at (5.1,3.9) {$A^P_Q$};
  \node at (10.2,3.9) {$B^P_Q$};
  \draw (-0.7,0)--(12.4,0);
  \draw[thick] (0,0)--(0,2.8) arc (180:90:6mm)--(12.2,3.4);
  \draw[thick] (0.2,0) node[below] {$\famX(P)$}--(0.2,2.8) arc (180:90:4mm)--(12.2,3.2);
  \draw[thick] (0.4,0)--(0.4,2.8) arc (180:90:2mm)--(12.2,3);
  \draw (1.6,0)--(1.6,5.2) (11.7,0)--(11.7,5.2);
  \draw[decorate,decoration={brace,amplitude=5pt,raise=3pt}] (1.6,5.2)--(11.7,5.2) node[midway,above=8pt] {$P$};
  \draw[dashed] (2.3,0)--(2.3,2.5)--(3.05,2.5) (8,0)--(8,2.5)--(3.05,2.5) node[fill=white] {$\famX^P$};
  \draw[thick] (3,0)--(3,1.4) arc (180:90:6mm)--(7.5,2);
  \draw[thick] (3.2,0) node[below] {$\famX^P(Q)$}--(3.2,1.4) arc (180:90:4mm)--(7.5,1.8);
  \draw[thick] (3.4,0)--(3.4,1.4) arc (180:90:2mm)--(7.5,1.6);
  \draw (4,0)--(4,4.4) (7,0)--(7,4.4);
  \draw[decorate,decoration={brace,amplitude=5pt,raise=3pt}] (4,4.4)--(7,4.4) node[midway,above=8pt] {$Q$};
  \draw[thick] (6.3,0) node[below] {$L(X^P_Q)$}--(6.3,2.3);
  \draw (4.6,0)--(4.6,4.3) (5.6,0)--(5.6,4.3);
  \draw[thick] (9,0) node[below] {$R(X^P_Q)$}--(9,1.7) arc (180:90:3mm)--(11.2,2);
  \draw (9.7,0)--(9.7,4.3) (10.7,0)--(10.7,4.3);
\end{tikzpicture}
\end{figure}

\noindent
In particular, $X^P_Q$ intersects the double-curves in $\famX^P(Q)$, $A^P_Q$ intersects the double-curves in $\famX(P)\cup\famX^P(Q)$, and $B^P_Q$ intersects the double-curves in $\famX(P)\cup\{X^P_Q\}$.
Let
\begin{equation*}
\famX_{k+1}=\famX\cup\bigcup_{P\in\famP}\famX^P\cup\bigcup_{P\in\famP}\bigl\{X^P_Q\colon Q\in\famP^P\bigr\},\quad\famP_{k+1}=\bigcup_{P\in\famP}\bigl\{A^P_Q,B^P_Q\colon Q\in\famP^P\bigr\}.
\end{equation*}
The conditions \ref{item:shape} and \ref{item:probe} clearly hold for $(\famX_{k+1},\famP_{k+1})$, and \ref{item:probe} for $(\famX_k,\famP_k)$ implies \ref{item:omega} for $(\famX_{k+1},\famP_{k+1})$.
To see that \ref{item:color} holds for $(\famX_{k+1},\famP_{k+1})$ and $k+1$, consider a proper coloring $\phi$ of $\famX_{k+1}$.
Let $\phi(X)$ denote the color of a double-curve $X\in\famX_{k+1}$ and $\phi(\famY)$ denote the set of colors used on a subset $\famY\subseteq\famX_{k+1}$.
By \ref{item:color} applied to $(\famX,\famP)$, there is a probe $P\in\famP$ such that $\lvert\phi(\famX(P))\rvert\geq k$.
By \ref{item:color} applied to $(\famX^P,\famP^P)$, there is a probe $Q\in\famP^P$ such that $\lvert\phi(\famX^P(Q))\rvert\geq k$.
Since $X^P_Q$ intersects the double-curves in $\famX^P(Q)$, we have $\phi(X^P_Q)\notin\phi(\famX^P(Q))$.
If $\phi(\famX(P))\neq\phi(\famX^P(Q))$, then $\famX_{k+1}(A^P_Q)=\famX(P)\cup\famX^P(Q)$ yields $\lvert\phi(\famX_{k+1}(A^P_Q))\rvert=\lvert\phi(\famX(P))\cup\phi(\famX^P(Q))\rvert\geq k+1$.
If $\phi(\famX(P))=\phi(\famX^P(Q))$, then $\famX_{k+1}(B^P_Q)=\famX(P)\cup\{X^P_Q\}$ and $\phi(X^P_Q)\notin\phi(\famX(P))$ yield $\lvert\phi(\famX_{k+1}(B^P_Q))\rvert=\lvert\phi(\famX(P))+1\rvert\geq k+1$.
This shows that \ref{item:color} holds for $(\famX_{k+1},\famP_{k+1})$ and $k+1$.
\end{proof}


\begin{thebibliography}{99}

\bibitem{Ack09}
\href{http://doi.org/10.1007/s00454-009-9143-9}{Eyal Ackerman, On the maximum number of edges in topological graphs with no four pairwise crossing edges, \emph{Discrete Comput. Geom.} 41~(3), 365--375, 2009}.

\bibitem{AAP+97}
\href{http://doi.org/10.1007/BF01196127}{Pankaj~K. Agarwal, Boris Aronov, János Pach, Richard Pollack, and Micha Sharir, Quasi-planar graphs have a linear number of edges, \emph{Combinatorica} 17~(1), 1--9, 1997}.

\bibitem{AG60}
\href{http://doi.org/10.7146/math.scand.a-10607}{Edgar Asplund and Branko Grünbaum, On a colouring problem, \emph{Math. Scand.} 8, 181--188, 1960}.

\bibitem{BMP-book}
\href{http://doi.org/10.1007/0-387-29929-7}{Peter Brass, William Moser, and János Pach, \emph{Research Problems in Discrete Geometry}, Springer, New York, 2005}.

\bibitem{Bur65}
James~P. Burling, \emph{On coloring problems of families of prototypes}, PhD thesis, University of Colorado, Boulder, 1965.

\bibitem{CSS}
\href{http://arxiv.org/abs/1609.00314}{Maria Chudnovsky, Alex Scott, and Paul Seymour, Induced subgraphs of graphs with large chromatic number. V. Chandeliers and strings, arXiv:1609.00314}.

\bibitem{FP12}
\href{http://doi.org/10.1016/j.ejc.2011.09.021}{Jacob Fox and János Pach, Coloring $K_k$-free intersection graphs of geometric objects in the plane, \emph{European J. Combin.} 33~(5), 853--866, 2012}.

\bibitem{FP14}
\href{http://doi.org/10.1017/S0963548313000412}{Jacob Fox and János Pach, Applications of a new separator theorem for string graphs, \emph{Combin. Prob. Comput.} 23~(1), 66--74, 2014}.

\bibitem{FPS13}
\href{http://doi.org/10.1137/110858586}{Jacob Fox, János Pach, and Andrew Suk, The number of edges in $k$-quasi-planar graphs, \emph{SIAM J. Discrete Math.} 27~(1), 550--561, 2013}.

\bibitem{Gya85}
\href{http://doi.org/10.1016/0012-365X(85)90044-5}{András Gyárfás, On the chromatic number of multiple interval graphs and overlap graphs, \emph{Discrete Math.} 55~(2), 161--166, 1985}.\periodsf
\href{http://doi.org/10.1016/0012-365X(86)90224-4}{Corrigendum: \emph{Discrete Math.} 62~(3), 333, 1986}.

\bibitem{Hen98}
\href{http://page.math.tu-berlin.de/~felsner/Diplomarbeiten/hendler.pdf}{Clemens Hendler, Schranken für Färbungs- und Cliquenüberdeckungszahl geometrisch repräsentierbarer Graphen (Bounds for chromatic and clique cover number of geometrically representable graphs), Master's thesis, Freie Universität Berlin, 1998}.

\bibitem{Kos88}
\href{http://math.nsc.ru/journals/ti/10/ti_10_0009.pdf}{Alexandr~V. Kostochka, O~verkhnikh otsenkakh khromaticheskogo chisla grafov (On upper bounds for the chromatic number of graphs), in: Vladimir~T. Dementyev (ed.), \emph{Modeli i~metody optimizacii}, vol.~10 of \emph{Trudy Inst. Mat.}, pp.~204--226, Akad. Nauk SSSR SO, Novosibirsk, 1988}.

\bibitem{Kos04}
\href{http://doi.org/10.1090/conm/342/06137}{Alexandr~V. Kostochka, Coloring intersection graphs of geometric figures with a given clique number, in: János Pach (ed.), \emph{Towards a Theory of Geometric Graphs}, vol.~342 of \emph{Contemp. Math.}, pp.~127--138, AMS, Providence, 2004}.

\bibitem{KK97}
\href{http://doi.org/10.1016/S0012-365X(96)00344-5}{Alexandr~V. Kostochka and Jan Kratochvíl, Covering and coloring polygon-circle graphs, \emph{Discrete Math.} 163~(1\nobreakdash--3), 299--305, 1997}.

\bibitem{KPW15}
\href{http://doi.org/10.1007/s00454-014-9640-3}{Tomasz Krawczyk, Arkadiusz Pawlik, and Bartosz Walczak, Coloring triangle-free rectangle overlap graphs with $O(\log\log n)$ colors, \emph{Discrete Comput. Geom.} 53~(1), 199--220, 2015}.

\bibitem{KW}
\href{http://doi.org/10.1007/s00493-016-3414-x}{Tomasz Krawczyk and Bartosz Walczak, On-line approach to off-line coloring problems on graphs with geometric representations, \emph{Combinatorica}, in press}.

\bibitem{LMPW14}
\href{http://doi.org/10.1007/s00454-014-9614-5}{Michał Lasoń, Piotr Micek, Arkadiusz Pawlik, and Bartosz Walczak, Coloring intersection graphs of arc-connected sets in the plane, \emph{Discrete Comput. Geom.} 52~(2), 399--415, 2014}.

\bibitem{Mat14}
\href{http://doi.org/10.1017/S0963548313000400}{Jiří Matoušek, Near-optimal separators in string graphs, \emph{Combin. Prob. Comput.} 23~(1), 135--139, 2014}.

\bibitem{McG96}
\href{http://doi.org/10.1016/0012-365X(95)00316-O}{Sean McGuinness, On bounding the chromatic number of L-graphs, \emph{Discrete Math.} 154~(1\nobreakdash--3), 179--187, 1996}.

\bibitem{McG00}
\href{http://doi.org/10.1007/PL00007228}{Sean McGuinness, Colouring arcwise connected sets in the plane~I, \emph{Graphs Combin.} 16~(4), 429--439, 2000}.

\bibitem{McG01}
\href{http://doi.org/10.1007/PL00007235}{Sean McGuinness, Colouring arcwise connected sets in the plane~II, \emph{Graphs Combin.} 17~(1), 135--148, 2001}.

\bibitem{PRT06}
\href{http://doi.org/10.1007/978-3-540-32439-3_12}{János Pach, Radoš Radoičić, and Géza Tóth, Relaxing planarity for topological graphs, in: Ervin Győri, Gyula~O.~H. Katona, and László Lovász (eds.), \emph{More Graphs, Sets and Numbers}, vol.~15 of \emph{Bolyai Soc. Math. Stud.}, pp.~285--300, Springer, Berlin, 2006}.

\bibitem{PSS96}
\href{http://doi.org/10.1007/BF02086610}{János Pach, Farhad Shahrokhi, and Mario Szegedy, Applications of the crossing number, \emph{Algorithmica} 16~(1), 111--117, 1996}.

\bibitem{PKK+13}
\href{http://doi.org/10.1007/s00454-013-9534-9}{Arkadiusz Pawlik, Jakub Kozik, Tomasz Krawczyk, Michał Lasoń, Piotr Micek, William~T. Trotter, and Bartosz Walczak, Triangle-free geometric intersection graphs with large chromatic number, \emph{Discrete Comput. Geom.} 50~(3), 714--726, 2013}.

\bibitem{PKK+14}
\href{http://doi.org/10.1016/j.jctb.2013.11.001}{Arkadiusz Pawlik, Jakub Kozik, Tomasz Krawczyk, Michał Lasoń, Piotr Micek, William~T. Trotter, and Bartosz Walczak, Triangle-free intersection graphs of line segments with large chromatic number, \emph{J. Combin. Theory Ser.~B} 105, 6--10, 2014}.

\bibitem{RW}
\href{http://arxiv.org/abs/1312.1559}{Alexandre Rok and Bartosz Walczak, Outerstring graphs are $\chi$-bounded, arXiv:1312.1559}.\periodsf
\href{http://doi.org/10.1145/2582112.2582115}{Preliminary version in: Siu-Wing Cheng and Olivier Devillers (eds.), \emph{30th Annual Symposium on Computational Geometry (SoCG 2014)}, pp.~136--143, ACM, New York, 2014}.

\bibitem{Suk14}
\href{http://doi.org/10.1007/s00493-014-2942-5}{Andrew Suk, Coloring intersection graphs of $x$-monotone curves in the plane, \emph{Combinatorica} 34~(4), 487--505, 2014}.

\bibitem{SW15}
\href{http://doi.org/10.1016/j.comgeo.2015.06.001}{Andrew Suk and Bartosz Walczak, New bounds on the maximum number of edges in $k$-quasi-planar graphs, \emph{Comput. Geom.} 50, 24--33, 2015}.

\bibitem{Val97}
\href{http://doi.org/10.1007/3-540-63938-1_63}{Pavel Valtr, Graph drawing with no $k$ pairwise crossing edges, in: Giuseppe Di Battista (ed.), \emph{5th International Symposium on Graph Drawing (GD 1997)}, vol.~1353 of \emph{Lecture Notes Comput. Sci.}, pp.~205--218, Springer, Berlin, 1997}.

\end{thebibliography}
\end{document}